\documentclass[11pt,twoside,reqno]{amsart}

\usepackage{microtype}
\usepackage[OT1]{fontenc}
\usepackage{type1cm}
\usepackage{amssymb}
\usepackage{xcolor}
\usepackage{mathrsfs}
\usepackage{tikz}
\usepackage{subcaption}

\usetikzlibrary{arrows,decorations.markings}

\numberwithin{equation}{section}

\theoremstyle{plain}
\newtheorem{theorem}{Theorem}[section]

\newtheorem{corollary}[theorem]{Corollary}
\newtheorem{proposition}[theorem]{Proposition}

\newtheorem{lemma}[theorem]{Lemma}

\theoremstyle{remark}
\newtheorem{remark}[theorem]{Remark}

\theoremstyle{definition}

\newcommand{\BB}{\mathcal{B}}

\newcommand{\PP}{\mathcal{P}}
\newcommand{\DD}{\mathcal{D}}
\newcommand{\MM}{\mathcal{M}}
\newcommand{\EE}{\mathcal{E}}
\newcommand{\CC}{\mathcal{C}}

\newcommand{\R}{\mathbb{R}}
\newcommand{\RP}{\mathbb{RP}^1}

\newcommand{\N}{\mathbb{N}}

\newcommand{\iii}{\mathtt{i}}
\newcommand{\jjj}{\mathtt{j}}
\newcommand{\kkk}{\mathtt{k}}

\newcommand{\eps}{\varepsilon}
\newcommand{\fii}{\varphi}

\newcommand{\A}{\mathsf{A}}

\newcommand{\dd}{\,\mathrm{d}}
\providecommand{\norm}[1]{\|#1\|}

\DeclareMathOperator{\dimh}{dim_H}

\DeclareMathOperator{\diml}{dim_L}

\DeclareMathOperator{\Var}{Var}

\begin{document}

\title{Birkhoff and Lyapunov spectra on planar self-affine sets}

\author{Bal\'azs B\'ar\'any}
\address[Bal\'azs B\'ar\'any]
        {Budapest University of Technology and Economics \\
         MTA-BME Stochastics Research Group \\
         P.O.\ Box 91 \\
         1521 Budapest \\
         Hungary}
\email{balubsheep@gmail.com}

\author{Thomas Jordan}
\address[Thomas Jordan]
        {School of Mathematics \\
         University of Bristol \\
         Bristol \\
         BS8 1SN \\
         United Kingdom}
\email{thomas.jordan@bristol.ac.uk}

\author{Antti K\"aenm\"aki}
\address[Antti K\"aenm\"aki]
        {Department of Physics and Mathematics \\
         University of Eastern Finland \\
         P.O.\ Box 111 \\
         FI-80101 Joensuu \\
         Finland}
\email{antti.kaenmaki@uef.fi}

\author{Micha\l{} Rams}
\address[Michal Rams]
        {Institute of Mathematics \\
         Polish Academy of Sciences \\
         S\'niadeckich 8 \\
         00-656 Warsaw \\
         Poland}
\email{rams@impan.pl}

%\thanks{}
%\date{\today}

\begin{abstract}
  Working on strongly irreducible planar self-affine sets satisfying the strong open set condition, we calculate the Birkhoff spectrum of continuous potentials and the Lyapunov spectrum.
\end{abstract}

\maketitle

\section{Introduction}

Let $\Sigma = \{ 1,\ldots,N \}^\N$ be the collection of all infinite words obtained from letters $\{ 1,\ldots,N \}$, and let $\sigma$ be the left-shift operator on $\Sigma$. The classical theorem of Birkhoff states that if $\mu$ is an ergodic $\sigma$-invariant probability measure, then $\frac1n \sum_{k=0}^{n-1} \Phi(\sigma^k\iii)$ converges to the average $\int_\Sigma \Phi \dd\mu$ of $\Phi$ for every $L^1$ potential $\Phi\colon\Sigma\to\R^M$ and for $\mu$-almost every $\iii \in \Sigma$. However, there are plenty of ergodic $\sigma$-invariant measures, for which the limit exists but converges to a different quantity. Furthermore, there are many words $\iii \in \Sigma$ which are not generic for any ergodic measure or even for which the limit $\lim_{n\to\infty}\frac1n \sum_{k=0}^{n-1} \Phi(\sigma^k\iii)$ does not exist at all. Thus, one may ask how rich is the set of points
$$
E_\Phi(\alpha) = \Bigl\{\iii\in\Sigma:\tfrac1n \sum_{k=0}^{n-1}\Phi(\sigma^k\iii)\to\alpha\text{ as }n\to\infty\Bigr\}.
$$
This 'richness' is usually calculated in terms of topological entropy or Hausdorff dimension. For a $\sigma$-invariant set $A\subset\Sigma$ we denote the topological entropy of $\sigma$ on $A$ by $h_{\rm top}(A)$ and the Hausdorff dimension of a set $A\subset\R^d$ by $\dimh(A)$. For the precise definitions, see Bowen~\cite{Bowennotbook} and Mattila~\cite{Mattila1995}. The functions $\alpha\mapsto h_{\rm top}(E_\Phi(\alpha))$ and $\alpha\mapsto\dimh(E_\Phi(\alpha))$ are respectively called the topological entropy spectrum and Hausdorff dimension spectrum of the Birfhoff averages of $\Phi$. For simplicity they are also called the Birkhoff spectrum of $\Phi$.

The topological entropy spectrum of Birkhoff averages has been intensely studied by several authors and is well understood for continuous potentials and vector valued continuous potentials; see Takens and Verbitskiy~\cite{TV}, Barreira, Saussol, and Schmeling~\cite{BarreiraSaussolSchmeling} and Feng, Fan, and Wu \cite{FanFengWu}, where \cite{FanFengWu} considers the endpoints of the spectrum.

In order to be able to study the Hausdorff dimension of the sets, where the limit of the Birkhoff average exists and takes a predefined value, we need to introduce a geometrical structure. The simplest example for such a geometrical structure is a self-similar iterated function system satisfying some separation condition.

Let $f_i\colon\R^d\to\R^d$ be contracting homeomorphisms for $i\in\{1,\ldots,N\}$. It is well known that there exists a non-empty, compact set $X\subset\R^d$ such that $X=\bigcup_{i=1}^Nf_i(X)$. Moreover, there is a H\"older continuous map $\pi\colon\Sigma\to X$, defined by $\pi(i_1i_2\cdots)=\lim_{n\to\infty}f_{i_1}\circ\cdots\circ f_{i_n}(0)$. If $f_i(X)\cap f_j(X)=\emptyset$ for $i\neq j$, then the inverses of $f_i$'s are well-defined. Therefore, $\pi$ is invertible and its inverse is also H\"older continuous. Thus, the mapping $T\colon X\to X$, $T(\pi(\iii))=\pi(\sigma\iii)$, is well-defined, and one can study Hausdorff dimension of the sets
$$
\pi E_\Phi(\alpha) = \Bigl\{x\in X:\lim_{n\to\infty}\tfrac1n \sum_{k=0}^{n-1}\Phi\circ\pi^{-1}(T^k(x))=\alpha\Bigr\}.
$$
Barreira and Saussol~\cite{BarreiraSaussol}, Feng, Lau, and Wu~\cite{FengLauWu}, and Olsen~\cite{Olsen} studied the setting where the $f_i$'s are conformal. In \cite{BarreiraSaussol}, the function $\Phi$ is H\"{o}lder continuous with codomain $\R$, the paper \cite{FengLauWu} considers the case where $\Phi$ may be continuous and addresses the endpoints of the spectrum, and \cite{Olsen} considers far more general $\Phi$ including the case where the codomain is $\R^d$.

In the conformal situation, a particular case of the Birkhoff averages, which is especially interesting from the dynamical systems' theory point of view, is the Lyapunov spectrum, which is obtained by taking the potential $\iii\mapsto-\log\|D_{\pi(\sigma\iii)}f_{\iii|_1}\|$. The Birkhoff average of this potential is called the Lyapunov exponent, and it is denoted by $\chi(\iii)$. The Lyapunov spectrum satisfies
$$
\dimh(\pi E_\chi(\alpha))=\frac{h_{\rm top}(E_{\chi}(\alpha))}{\alpha},
$$
where $E_\chi(\alpha)$ is the set of points $\iii\in\Sigma$ for which $\chi(\iii)=\alpha$.

The goal of this paper is to investigate the non-conformal situation. We consider the linear case, that is, the mappings $f_i$ are assumed to satisfy $f_i(x)=A_ix+v_i$, where $v_i\in\R^d$ and $A_i\in GL_d(\R)$ so that $\|A_i\|<1$ for all $i\in\{1,\ldots,N\}$. In this case, we have at the moment a very limited knowledge on the Hausdorff spectrum of Birkhoff averages. Barral and Mensi \cite{BarralMensi} studied the Birkhoff spectrum on Bedford-McMullen carpets. This result was generalised for Gatzouras-Lalley carpets by Reeve~\cite{Reeve}. Moreover, Jordan and Simon \cite{JordanSimon} studied the case of planar affine iterated function systems with diagonal matrices for generic translation vectors $v_i$. As far as we are aware of, the only known result about non-diagonal matrices comes from K\"aenm\"aki and Reeve~\cite{KaenmakiReeve2014}, who investigated irreducible matrices with generic translation vectors (see Section~\ref{sec:preli} for the precise definition of irreducibility).

In the $d$-dimensional non-conformal situation, the Lyapunov exponents are in general not given by a Birkhoff average of some potential. However, under some domination conditions, it is still true. Barreira and Gelfert \cite{BarreiraGelfert} studied the topological entropy spectrum of the maximal Lyapunov exponent for non-linear, planar systems under certain domination condition, and Feng~\cite{Feng2009} studied it in the case of non-negative matrices in higher dimensions. Later, Feng and Huang~\cite{FengHuang} calculated the topological entropy spectrum of maximal Lyapunov exponent in high generality. D\'iaz, Gelfert and Rams \cite{DiazGelfertRams} investigated the non-dominated case for $d=2$ providing results on topoligical entropy spectrum of the difference of the Lyapunov exponents and showing that the equality of Lyapunov exponents happens with positive but not maximal topological entropy for generic non-dominated systems.

In the main results of this paper, we calculate the Birkhoff and Lyapunov spectra for planar affine iterated function systems satisfying strong irreducibility and the strong open set condition. After introducing some notation and preliminaries in Section~\ref{subsec:not}, we formulate the results in Section~\ref{subsec:main}. The key idea in the proof of the Birkhoff spectrum is to first calculate it on a family of affine iterated function systems which satisfy the dominated splitting condition. This is done in Section~\ref{sec:domin}. The general case can then be approximated by systems satisfying the dominated splitting; see Section~\ref{sec:subsys}. This idea, with suitable modifications, is also applied in the proof of the Lyapunov spectrum; see Section~\ref{sec:lyap}. Finally, we partially handle the boundary case in Section~\ref{sec:lyapdiag}.

\section{Main results} \label{sec:preli}

Throughout the rest of the paper, we restrict ourselves to planar systems.

\subsection{Notation}\label{subsec:not}

A tuple $\Theta = (A_1+v_1,\ldots,A_N+v_N)$ of contractive invertible affine self-maps on $\R^2$ is called an \emph{affine iterated function system (affine IFS)}. The associated tuple of matrices $(A_1,\ldots,A_N)$ is therefore an element of $GL_2(\R)^N$ and satisfies $\max_{i \in \{1,\ldots,N\}}\|A_i\|<1$. As already mentioned in the introduction, there exists a unique non-empty compact set $X \subset \R^2$ such that
\begin{equation*}
  X = \bigcup_{i=1}^N (A_i+v_i)(X).
\end{equation*}
In this case, the set $X$ is called a \emph{self-affine set}. We say that $\Theta$ satisfies a \emph{strong open set condition (SOSC)} if there exists an open set $U \subset \R^2$ intersecting $X$ such that the union $\bigcup_{i=1}^N (A_i+v_i)(U)$ is pairwise disjoint and is contained in $U$. Furthermore, $\Theta$ satisfies the \emph{strong separation condition (SSC)} if $(A_i+v_i)(X) \cap (A_j+v_j)(X) = \emptyset$ whenever $i \ne j$.

We say that $\A = (A_1,\ldots,A_N) \in GL_2(\R)^N$ is \emph{irreducible} if there does not exist a $1$-dimensional linear subspace $V$ such that $A_iV=V$ for all $i \in \{ 1,\ldots,N \}$; otherwise $\A$ is \emph{reducible}. The tuple $\A$ is \emph{strongly irreducible} if there does not exist a finite union of $1$-dimensional subspaces, $V$, such that $A_iV=V$ for all $i\in\{1,\ldots,N\}$. In a reducible tuple $\A$, all the matrices are simultaneously upper triangular in some basis. For tuples with more than one element, strong irreducibility is a generic property.

We say that $\A = (A_1,\ldots,A_N) \in GL_2(\R)^N$ is \emph{normalized} if $|\det(A_i)|=1$ for $i\in\{1,\ldots,N\}$. A matrix $A$ is called \emph{hyperbolic} if it has two real eigenvalues with different absolute value, \emph{elliptic} if it has two complex eigenvalues, and \emph{parabolic} if it is neither elliptic nor hyperbolic. The subgroup generated by $\mathsf{A}$ is \emph{relatively compact} if and only if the generated subgroup contains only elliptic matrices or orthogonal parabolic matrices. Otherwise we call it non-compact. For a hyperbolic matrix $A$, let $s(A)$ be the eigenspace corresponding to the eigenvalue with the largest absolute value and let $u(A)$ be the eigenspace corresponding to the smallest absolute valued eigenvalue.

Let $\Sigma = \{ 1,\ldots,N \}^\N$ be the collection of all infinite words obtained from integers $\{ 1,\ldots,N \}$. If $\iii = i_1i_2\cdots \in \Sigma$, then we define $\iii|_n = i_1 \cdots i_n$ for all $n \in \N$. The empty word $\iii|_0$ is denoted by $\varnothing$. Define $\Sigma_n = \{ \iii|_n : \iii \in \Sigma \}$ for all $n \in \N$ and $\Sigma_* = \bigcup_{n \in \N} \Sigma_n \cup \{ \varnothing \}$. Thus $\Sigma_*$ is the collection of all finite words. %For each $n \in \N$, the set $\Sigma^{(n)}$ is the collection of all infinite word obtained from length $n$ words. Of course, as a set it is the same as $\Sigma$, only the indexing is different. For example, writing $\Sigma^{(n)}_m = (\Sigma^{(n)})_m$, we have $\Sigma_m^{(n)} = \Sigma_{mn}$ for all $m,n \in \N$. Let $\mathfrak{b}_n$ be the canonical bijection $\Sigma^{(n)} \to \Sigma$.
The length of $\iii \in \Sigma_* \cup \Sigma$ is denoted by $|\iii|$. The longest common prefix of $\iii,\jjj \in \Sigma_* \cup \Sigma$ is denoted by $\iii \wedge \jjj$. The concatenation of two words $\iii \in \Sigma_*$ and $\jjj \in \Sigma_* \cup \Sigma$ is denoted by $\iii\jjj$. Let $\sigma$ be the left shift operator defined by $\sigma\iii = i_2i_3\cdots$ for all $\iii = i_1i_2\cdots \in \Sigma$. If $\iii \in \Sigma_n$ for some $n$, then we set $[\iii] = \{ \jjj \in \Sigma : \jjj|_n = \iii \}$. The set $[\iii]$ is called a \emph{cylinder set}. The \emph{shift space} $\Sigma$ is compact in the topology generated by the cylinder sets. Moreover, the cylinder sets are open and closed in this topology and they generate the Borel $\sigma$-algebra.

Write $A_\iii = A_{i_1} \cdots A_{i_n}$ for all $\iii = i_1 \cdots i_n \in \Sigma_n$ and $n \in \N$. The canonical projection $\pi\colon \Sigma \to X$ is defined by $\pi(\iii) = \sum_{n=1}^\infty A_{\iii|_{n-1}} v_{i_n}$ for all $\iii = i_1i_2\cdots \in \Sigma$. It is easy to see that $\pi(\Sigma) = X$. If $\mu$ is a measure on $\Sigma$, then we denote the pushforward measure of $\mu$ under $\pi$ by $\pi_*\mu=\mu\circ\pi^{-1}$.

Let $\MM_\sigma(\Sigma)$ denote the collection of all $\sigma$-invariant probability measures on $\Sigma$ and $\EE_\sigma(\Sigma)$ be the collection of ergodic elements in $\MM_\sigma(\Sigma)$. Let $\mu \in \MM_\sigma(\Sigma)$ and recall that the \emph{Kolmogorov-Sinai entropy} of $\mu$ and $\sigma$ is
\begin{equation*}
  h(\mu) := h(\mu,\sigma) = -\lim_{n \to \infty} \tfrac{1}{n} \sum_{\iii \in \Sigma_n} \mu([\iii]) \log\mu([\iii]).
\end{equation*}  A probability measure $\mu$ on $(\Sigma,\sigma)$ is \emph{Bernoulli} if there exist a probability vector $(p_1,\ldots,p_N)$ such that
$$
  \mu([\iii])=p_{i_1}\cdots p_{i_{n}}
$$
for all $\iii=i_1\cdots i_n\in\Sigma_n$ and $n \in \N$. It is well-known that Bernoulli measures are ergodic. We say that $\mu \in \MM_{\sigma^n}(\Sigma)$ is an \emph{$n$-step Bernoulli} if it is a Bernoulli measure on $(\Sigma,\sigma^n)$. In this case, we write
\begin{equation}\label{nbern}
\tilde{\mu}=\tfrac{1}{n}\sum_{k=0}^{n-1}\mu\circ\sigma^{-k}
\end{equation}
and note that $\tilde{\mu}\in\EE_\sigma(\Sigma)$, $h(\mu,\sigma^n)=nh(\tilde{\mu},\sigma)$, and $\int_\Sigma S_nf\dd\mu=n\int_\Sigma f\dd\tilde{\mu}$ for all continuous $f\colon\Sigma\to\R$, where $S_nf=\sum_{k=0}^{n-1}f\circ\sigma^{k}$ is the \emph{Birkhoff sum}.
In addition, if $\mathsf{A} = (A_1,\ldots,A_N) \in GL_2(\R)^N$ and $\mu \in \MM_\sigma(\Sigma)$, then we define the \emph{Lyapunov exponents} of $\A$ with respect to $\mu$ and $\sigma$ to be
\begin{equation} \label{eq:lyapunov-def}
\begin{split}
  \chi_1(\mu):=\chi_1(\mu,\sigma) &= -\lim_{n \to \infty} \tfrac{1}{n} \int_\Sigma \log\|A_{\iii|_n}\| \dd\mu(\iii), \\
  \chi_2(\mu):=\chi_2(\mu,\sigma) &= -\lim_{n \to \infty} \tfrac{1}{n} \int_\Sigma \log\|A_{\iii|_n}^{-1}\|^{-1} \dd\mu(\iii).
\end{split}
\end{equation}
We define a function $\chi \colon \MM_\sigma(\Sigma) \to \R^2$ by setting $\chi(\mu)=(\chi_1(\mu),\chi_2(\mu))$ for all $\mu \in \MM_\sigma(\Sigma)$. The \emph{Lyapunov exponents} at $\iii \in \Sigma$ are defined by
\begin{align*}
  \underline{\chi}_1(\iii) &= -\liminf_{n\to\infty}\tfrac{1}{n}\log\|A_{\iii|_n}\|, \\
  \overline{\chi}_1(\iii) &= -\limsup_{n\to\infty}\tfrac{1}{n}\log\|A_{\iii|_n}\|, \\
  \underline{\chi}_2(\iii) &= -\liminf_{n\to\infty}\tfrac{1}{n}\log\|A_{\iii|_n}^{-1}\|^{-1}, \\
  \overline{\chi}_2(\iii) &= -\limsup_{n\to\infty}\tfrac{1}{n}\log\|A_{\iii|_n}^{-1}\|^{-1}.
\end{align*}
If $\underline{\chi}_k(\iii)=\overline{\chi}_k(\iii)$, then we write $\chi_k(\iii)$ for the common value. Recall that, by Kingman's subadditive ergodic theorem, if $\mu \in \EE_\sigma(\Sigma)$, then $(\chi_1(\iii),\chi_2(\iii)) = \chi(\mu)$ for $\mu$-almost all $\iii \in \Sigma$. Since any two different ergodic measures are mutually singular, it is an interesting question to try to determine the size of a level set
$$
  E_\chi(\alpha) = \{\iii\in\Sigma : (\chi_1(\iii),\chi_2(\iii))=\alpha\}
$$
for a given value $\alpha$ from the set
\begin{equation*}
  \PP(\chi) = \left\{\alpha\in\R^2 : \iii\in\Sigma\text{ and }(\chi_1(\iii),\chi_2(\iii))=\alpha\right\}.
\end{equation*}
The \emph{Lyapunov dimension} of $\mu \in \MM_\sigma(\Sigma)$ is defined to be
\begin{equation*}
  \diml(\mu) = \min\biggl\{ \frac{h(\mu)}{\chi_1(\mu)}, 1+\frac{h(\mu)-\chi_1(\mu)}{\chi_2(\mu)} \biggr\}.
\end{equation*}

We say that a potential $\fii \colon \Sigma_* \to [0,\infty)$ is submultiplicative if
\begin{equation*}
  \fii(\iii\jjj) \le \fii(\iii)\fii(\jjj)
\end{equation*}
for all $\iii,\jjj \in \Sigma_*$. Recall that if $A \in GL_2(\R)$, then the lengths of the semiaxes of the ellipse $A(B(0,1))$ are given by $\|A\|$ and $\|A^{-1}\|^{-1}$. We define the \emph{singular value function} with parameter $s$ to be
\begin{equation*}
  \fii^s(A) =
  \begin{cases}
    \|A\|^s, &\text{if } 0 \le s < 1, \\
    \|A\| \|A^{-1}\|^{-(s-1)}, &\text{if } 1 \le s < 2, \\
    |\det(A)|^{s/2}, &\text{if } 2 \le s < \infty.
  \end{cases}
\end{equation*}
Intuitively, $\fii^s(A)$ represents a measurement of the $s$-dimensional volume of the image of the unit ball under $A$. Since $\fii^s(A) = |\det(A)|^{s-1}\|A\|^{2-s}$ for $1 \le s < 2$, we see that the function $\iii \mapsto \fii^s(A_\iii)$ defined on $\Sigma_*$ is submultiplicative. By a slight abuse of notation, if the tuple $(A_1,\ldots,A_N) \in GL_d(\R)^N$ is clear from the content, we refer to the function $\iii \mapsto \fii^s(A_\iii)$ also by $\fii^s$.

We also need a more general class of potentials. Define
$$
\psi^q(\iii)=\|A_\iii\|^{q_1}\|A_\iii^{-1}\|^{-q_2}
$$
for all $q = (q_1,q_2) \in \R^2$. This is a generalisation of $\varphi^s$: by taking
\begin{equation}\label{eq:s'}
s'(s)=\begin{cases}
     (s,0),     &\text{if } 0\le s<1, \\
     (1,s-1), &\text{if } 1\le s<2, \\
     (s/2,s/2),   &\text{if } 2\le s < \infty,
   \end{cases}
\end{equation}
it is easy to see that $\psi^{s'(s)}=\varphi^s.$

%Given a pair of potentials $\varphi_1, \varphi_2 \colon \Sigma_* \rightarrow [0,\infty)$, we let $\varphi_1 \cdot \varphi_2$ denote the potential defined by $\iii \mapsto \varphi_1(\iii) \varphi_2(\iii)$ for all $\iii \in \Sigma_*$ {\color{red}[Needed?]}.
If $\Phi \colon \Sigma \rightarrow \R$ is continuous, then we define the \emph{pressure} by
$$
  P(\log\varphi^s+\Phi)=\lim_{n\to\infty}\tfrac{1}{n}\log\sum_{\iii\in\Sigma_n}\varphi^s(\iii)\sup_{\jjj\in[\iii]}\exp(S_n\Phi(\jjj)),
$$
where the limit exists by subadditivity. For $q=(q_1,q_2)\in\R^2$, one defines the pressure by
$$
  P(\log\psi^q)=\lim_{n\to\infty}\frac{1}{n}\log\sum_{\iii\in\Sigma_n}\psi^q(\iii).
$$
Note that if $q_1 \ge q_2$, then $\psi^q$ is submultiplicative, and if $q_1 < q_2$, then $\psi^q$ is supermultiplicative, which also guarantees the existence of the limit.

Given a continuous potential $\Phi \colon \Sigma \to \R^M$, we let
\begin{align*}
  \PP(\Phi) &= \{\alpha\in\R^M : \iii\in\Sigma\text{ and }\lim_{n\to\infty}\tfrac1n S_n\Phi(\iii)=\alpha\} \\
  &= \{\alpha\in\R^M : \mu \in \MM_\sigma(\Sigma) \text{ and } \int_\Sigma\Phi \dd\mu = \alpha\}
\end{align*}
be the set of possible values of Birkhoff averages. Note that the equality above follows from \cite[Theorem~2.1.6 and Remark~2.1.15]{PrzytyckiUrbanski}. Write
\begin{equation*}
  E_\Phi(\alpha) = \{\iii\in\Sigma : \lim_{n\to\infty}\tfrac{1}{n}S_n\Phi(\iii) = \alpha \}
\end{equation*}
for all $\alpha\in\R^M$.

We use Bowen's definition \cite{Bowennotbook} of topological entropy which is defined for non-compact and non-invariant sets.
It follows from Takens and Verbitskiy \cite[Theorem~5.1]{TV} that
\begin{equation}\label{entropy}
  h_\mathrm{top}(E_\Phi(\alpha))=\lim_{\varepsilon\downarrow 0}\liminf_{n\to\infty}\tfrac{1}{n}\log\#\{\iii|_n\in\Sigma_n : |\alpha-\tfrac{1}{n}S_n\Phi(\iii)|<\varepsilon\},
  \end{equation}
 where $|\cdot|$ denotes the Euclidean norm in $\R^M$.
Note that the result is for continuous potentials having range in $\R$ but it easily extends to the case where the range is in $\R^M$. Recall that, by Bowen~\cite{Bowennotbook}, we have
$$
  h_\mathrm{top}(E)\geq\sup\{h(\mu) :\mu\in\MM_\sigma(\Sigma)\text{ and }\mu(E)=1\}.
$$
The topological closure of a set $A$ is denoted by $\overline{A}$, the boundary by $\partial A$, and the interior by $A^o$.

\subsection{Main theorems}\label{subsec:main}

We are now ready to formulate our main theorems. The results determine the Hausdorff dimensions of the canonical projections of $E_\Phi(\alpha)$ and $E_\chi(\alpha)$.

\begin{theorem}[Birkhoff spectrum]\label{thm:mainbirkhoffint}
Let $(A_1+v_1,\ldots,A_N+v_N)$ be an affine IFS on $\R^2$ satisfying the SOSC and $\Phi\colon\Sigma\to\R^M$ be a continuous potential. If $(A_1,\ldots,A_N) \in GL_2(\R)^N$ is strongly irreducible such that the generated subgroup of the normalized matrices is not relatively compact, then $\PP(\Phi)$ is compact and convex, and
\begin{align*}
  \dimh(\pi E_{\Phi}(\alpha)) &= \sup\{\diml(\mu) :\mu\in\MM_\sigma(\Sigma)\text{ and }\int_\Sigma\Phi \dd\mu = \alpha\} \\
	&= \sup\{\diml(\mu) : \mu\in\EE_\sigma(\Sigma)\text{ and }\int_\Sigma\Phi \dd\mu = \alpha\} \\
  &= \sup\{ s \ge 0 : \inf_{q \in \R^M}P(\log \varphi^s +\left\langle q, \Phi - \alpha \right\rangle) \geq 0 \}
\end{align*}
for all $\alpha\in\PP(\Phi)^o \subset \R^M$. Furthermore, the function $\alpha \mapsto \dimh(\pi E_{\Phi}(\alpha))$ is continuous on $\PP(\Phi)^o$.
\end{theorem}

\begin{theorem}[Lyapunov spectrum]\label{thm:mainlyapunovint}
Let $(A_1+v_1,\ldots,A_N+v_N)$ be an affine IFS on $\R^2$ satisfying the SOSC. If $(A_1,\ldots,A_N) \in GL_2(\R)^N$ is strongly irreducible such that the generated subgroup of the normalized matrices is not relatively compact, then
\begin{align*}
  \dimh(\pi E_\chi(\alpha)) &= \sup\{\diml(\mu) : \mu\in\MM_\sigma(\Sigma)\text{ and }\chi(\mu)=\alpha\} \\
	&= \sup\{\diml(\mu) : \mu\in\EE_\sigma(\Sigma)\text{ and }\chi(\mu)=\alpha\} \\
  &= \sup\{ s \ge 0 : \inf_{q \in \R^2}\{P(\log\psi^{s'(s)-q})-\langle q,\alpha \rangle \} \geq 0 \} \\
  &= \min\biggl\{\frac{h_\mathrm{top}(E_\chi(\alpha))}{\alpha_1},1+\frac{h_\mathrm{top}(E_\chi(\alpha))-\alpha_1}{\alpha_2}\biggr\}
\end{align*}
for all $\alpha=(\alpha_1,\alpha_2)\in\PP(\chi)^o \subset \R^2$, where $s'\colon\R_+ \to \R^2$ is defined in \eqref{eq:s'}. Furthermore, if $\alpha\in\PP(\chi)\cap\{(\alpha_1,\alpha_2)\in\R^2:\alpha_1=\alpha_2\}$, then
\begin{equation*}
  \dimh(\pi E_\chi(\alpha))=\lim_{\eps\downarrow 0}\sup\{\diml(\mu) : \mu\in\MM_\sigma(\Sigma) \text{ and }|\chi(\mu)-\alpha|\leq\eps\}.
\end{equation*}
Moreover, $\PP(\chi)$ is convex and compact, and the function $\alpha \mapsto \dimh(\pi E_\chi(\alpha))$ is continuous on $\PP(\chi)^o\cup(\PP(\chi)\cap\{(\alpha_1,\alpha_2) \in \R^2 : \alpha_1=\alpha_2\})$.
\end{theorem}

\subsection{Further discussion}

If the generated subgroup of the normalised matrices is compact, then we have $\chi_{1}(\mu)=\chi_2(\mu)$ for any invariant measure $\mu$. Therefore, Theorem \ref{thm:mainbirkhoffint} is still true in this setting, but rather than following the proof in this paper, it is a simpler approach to use the standard conformal methods. Furthermore, it seems possible to discard the SOSC and consider the overlapping case; see Remark \ref{rem:hochman-rapaport}.

In the conformal setting, both the Lyapunov spectrum and the Birkhoff spectrum have a unique maximum which corresponds to the integral of the measure of maximal dimension. The non-conformal case is more involved. For the Birkhoff spectrum, in the setting of Theorem \ref{thm:mainbirkhoffint}, it is possible to show that there will be at most one maximum. If we let $s=\dimh(X)$ and $\mu$ be the unique ergodic measure with Lyapunov dimension $s$ then for any $\alpha\in\PP(\Phi)^o \subset \R^M$, $\dimh(\pi E_{\Phi}(\alpha))=s$ if and only if $\int_\Sigma\Phi \dd\mu = \alpha$. To see this first note that if $\alpha\in\PP(\Phi)^o \subset \R^M$ and $\int_\Sigma\Phi \dd\mu = \alpha$ then $\dimh(\pi E_{\Phi}(\alpha))=s$ is an immediate consequence of Theorem \ref{thm:mainbirkhoffint}. On the other hand, if $\dimh(\pi E_{\Phi}(\alpha))=s$ then there exists a sequence of invariant measures $\mu_n$ with $\int_\Sigma\Phi \dd\mu_n = \alpha$ for each $n\in\N$ and $\lim_{n\to\infty}\diml(\mu_n)=s$. Any weak$^*$ limit of these measures must be invariant and have integral $\alpha$. By K\"aenm\"aki and Reeve \cite[Proposition 6.8]{KaenmakiReeve2014}, it follows that the Lyapunov dimension must be at least $s$, and so the weak$^*$ limit must be the unique measure of maximal Lyapunov dimension $\mu$ and $\int_\Sigma\Phi \dd\mu = \alpha$.

In Theorem \ref{thm:mainbirkhoffint}, in the case where the self-affine system is dominated and the function $\Phi$ is H\"{o}lder, it is possible to show that the spectrum varies analytically away from integer values. The argument would follow the one given in Barreira and Saussol \cite{BarreiraSaussol} with an adaptation for the case of higher dimensions. The same argument also holds in Theorem \ref{thm:mainlyapunovint} when the system is dominated.

In the situation, where the system is not strongly irreducible, the results are no longer always true; for example, see Reeve \cite{Reeve} where the author considers self-affine carpets. In the diagonal case, it would be possible to combine Hochman \cite{Hochman2014} and Jordan and Simon \cite{JordanSimon} to get results for a large class of systems based on the dimension of projections of measures onto the $x$-axis and $y$-axis.

In addition to the obtained results it is possible to adapt our methods to obtain analogous results for sets with a similar definition to $E_{\chi}(\alpha)$. For example, under the same assumptions as for Theorem \ref{thm:mainlyapunovint}, the sets
$$
E_{\chi,s}(\alpha)=\{\iii\in \Sigma:\lim_{n\to\infty}\tfrac{1}{n}\log\fii^s(\iii)=\alpha\}
$$
could be considered for $s>0$. The method would be very similar to first show the results for dominated system and then deduce the full result by a slight adpatation of Lemma \ref{lem:pressapprox}. This approach would also work for the sets
$$
E_{\underline{\chi}}(\alpha)=\{\iii\in\Sigma:\chi_2(\iii)=\alpha\}.
$$

Lemma~\ref{lem:pressapprox} also gives a simple proof for the continuity of the pressure $P$ with respect to the matrix tuples in the two-dimensional case; see Feng and Shmerkin~\cite{FengShmerkin2014}. Namely, $P$ is upper semicontinuous since, by definition, it is an infimum of continuous functions. Lemma~\ref{lem:pressapprox} implies that it is a supremum of continuous functions, so it is also lower semicontinuous and hence continuous.

In Theorem \ref{thm:mainbirkhoffint}, if $\PP(\Phi)$ has empty interior, then it means that $\PP(\Phi)$ will be contained in some $k$-dimensional hyperplane where $k<m$. If $\PP(\Phi)$ has nonempty interior in this hyperplane then our results still apply. In Theorem \ref{thm:mainlyapunovint}, if $\PP(\chi)$ has empty interior, then, since it is a convex set, it is either contained in a one dimensional hyperplane or a point. In the case, where it is a one-dimensional hyperplane again we can work with the interior when restricted to this line. The proofs for these cases are virtually identical to the given ones, hence, we omit them.

As an important part of the study of the Birkhoff averages, one is also interested in the set
\begin{equation*}
  D_\Phi = \Bigl\{ \iii \in \Sigma : \lim_{n\to\infty}\tfrac{1}{n}\sum_{k=0}^{n-1} \Phi(\sigma^k\iii) \text{ does not exist} \Bigr\}.
\end{equation*}
In the conformal case, it holds that $\dimh(\pi D_\Phi) = \dimh(X)$; see Barreira and Schmeling \cite{BarreiraSchmeling2000}. It would be interesting to see if the same holds also in our setting. The lower bounds we obtain rely on finding the dimension of Bernoulli measures, which cannot help for this question. The same problem appears also in the question whether $\dimh(\pi G_\mu) = \diml(\mu)$, where
\begin{equation*}
  G_\mu = \Bigl\{ \iii \in \Sigma : \tfrac{1}{n}\sum_{k=0}^{n-1} \delta_{\sigma^k\iii} \text{ converges weakly to } \mu \Bigr\}
\end{equation*}
for all $\mu \in \MM_\sigma(\Sigma)$.

For shift spaces the topological entropy spectrum for Birkhoff averages of a continuous function is concave and can be written as a Legendre transform of a suitable pressure function, see Takens and Verbitskiy \cite[Section 6]{TV}. For an expanding interval Markov map the Lyapunov spectrum is a simple transform of the topological entropy spectrum  for a suitable potential. This means that the Lyapunov spectrum in this setting is a straightforward transform of a concave function (however the transform does not necessary preserve the concavity). In the setting of Theorem \ref{thm:mainlyapunovint}, the function $\alpha \mapsto h_{\rm{top}}(E_{\xi(\alpha)})$ will still be concave (see Proposition \ref{propconcave}) and so the Lyapunov spectrum is again a transform of a concave function but again it will not necessarily preserve concavity. It should be possible by careful case by case analysis to determine certain regions for which the Birkhoff spectrum and Lyapunov spectrum are concave but in general they will not be concave.

Finally, it is a natural question to ask whether the results could be extended to higher dimensions. There are two stumbling blocks to this. Firstly the result on the dimension of Bernoulli measures for strongly irreducible systems in B\'ar\'any, Hochman, and Rapaport \cite{bhr2017selfaffine} is only proved in the dimension two (see Hochman and Rapaport \cite[Section 1.6]{hr2019selfaffine} for discussion). Also the approximation via dominated systems becomes much more problematic in higher dimensions.

%%%%%%%%%%%%%%%%%%%%%%%%%%

\section{Upper bound}\label{sec:preli2}

Throughout the paper, our only assumption about the potential $\Phi$ is that it is continuous. A standard lemma gives bounds on the variation of $S_n\Phi$ inside $n$th level cylinders.

\begin{lemma} \label{lem:variance}
For any continuous $\Phi \colon \Sigma\to \R^M$,
\[
\lim_{n\to\infty} \tfrac 1n \max_{\kkk\in\Sigma_n} \max_{\iii, \jjj\in [\kkk]}|S_n\Phi(\iii)-S_n\Phi(\jjj)|=0.
\]
\end{lemma}

\begin{proof}
Denote by $\Var_n(\Phi)$ the $n$th variation of the potential $\Phi$, i.e.
\[
\Var_n(\Phi) = \max_{\kkk\in\Sigma_n}\max_{\iii, \jjj\in [\kkk]}|\Phi(\iii)-\Phi(\jjj)|.
\]
By the compactness of $\Sigma$, we have $\Var_n(\Phi)\to 0$. The assertion follows since
\[
\tfrac 1n \max_{\kkk\in\Sigma_n} \max_{\iii, \jjj\in [\kkk]}|S_n\Phi(\iii)-S_n\Phi(\jjj)| \leq \tfrac 1n \sum_{k=1}^n \Var_k(\Phi)
\]
for all $n \in \N$.
\end{proof}

The next proposition gives an upper bound for the dimension of $\pi E_{\Phi}(\alpha)$ by means of the pressure. Its proof is a standard covering argument.

\begin{proposition}\label{thm:birkhoffupper}
  Let $(A_1+v_1,\ldots,A_N+v_N)$ be an affine IFS on $\R^2$ and $\Phi \colon \Sigma \to \R^M$ be a continuous potential. Then
  $$
  \dimh(\pi E_{\Phi}(\alpha))\leq\sup\{ s \ge 0 : \inf_{q \in \R^M}P(\log \varphi^s +\langle q, \Phi - \alpha \rangle) \geq 0 \}
  $$
  for all $\alpha\in\PP(\Phi)^o$.
\end{proposition}

\begin{proof}
Observe that
$$
  E_{\Phi}(\alpha)\subset\bigcap_{r=1}^\infty\bigcup_{n=1}^\infty\bigcap_{m=n}^\infty\bigcup_{\iii\in D_{m,r}}[\iii],
$$
where
$$
  D_{m,r} = \{\iii\in\Sigma_m : \jjj\in[\iii]\text{ and }|\tfrac{1}{m}S_m\Phi(\jjj)-\alpha|<\tfrac{1}{r}\}.
$$
By Lemma \ref{lem:variance}, there exists $\zeta>0$ such that $|\tfrac{1}{m}S_m\Phi(\jjj)-\alpha|<\tfrac{2}{r}$ for all $m\geq -\log\zeta$, $\iii\in\Sigma_m$, and $\jjj\in [\iii]$. Therefore,
\begin{equation} \label{eqn:largem}
  -\frac{2m|q|}{r}\leq\langle q,S_m\Phi(\jjj)-m\alpha\rangle\leq\sup_{\jjj\in[\iii]}\langle q,S_m\Phi(\jjj)-m\alpha\rangle
\end{equation}
for all $\iii\in D_{m,r}$ and $m\geq-\log\zeta$.

Let $s_0(\alpha) = \sup\{ s \ge 0 : \inf_{q \in \R^M}P(\log \varphi^s +\left\langle q, \Phi - \alpha \right\rangle) \geq 0 \}$ and choose $s > s_0(\alpha)$. Thus there exists $q=q(\alpha,s)$ such that $P(\log\varphi^s +\left\langle q, \Phi - \alpha \right\rangle)<0$. Write $P=P(\log\varphi^s +\left\langle q, \Phi - \alpha \right\rangle)$ and let $\gamma>0$ be such that
\begin{equation}\label{eq:sumbound}
  \sum_{\iii \in \Sigma_m}\varphi^s(\iii)\exp(\sup_{\jjj\in[\iii]}\langle q,S_m\Phi(\jjj)-m\alpha\rangle)<e^{mP/2}
\end{equation}
for all $m\geq-\log\gamma$.
As $\|{A_{\iii}}\| \leq \lambda^{|\iii|}$ for some $\lambda<1$, we have
\begin{equation}\label{eq:hyper}
\varphi^{s+c}(\iii) \leq \varphi^s(\iii) e^{c|\iii|\log \lambda}
\end{equation}
for all $c\geq0$.
Hence, for every $\delta<\min\{\gamma,\zeta\}$, we have by \eqref{eq:hyper}, \eqref{eqn:largem} and \eqref{eq:sumbound}
\begin{align*}
\mathcal{H}^{s-2|q|/r\log\lambda}_\delta(\pi E_\Phi(\alpha))&\leq\sum_{m=\lceil-\log\delta\rceil}^{\infty}\sum_{\iii\in D_{m,r}} \varphi^{s-2|q|/r\log\lambda}(\iii)\\
&\leq\sum_{m=\lceil-\log\delta\rceil}^{\infty}\sum_{\iii\in D_{m,r}} \varphi^s(\iii)e^{-2m|q|/r}\\
&\leq\sum_{m=\lceil-\log\delta\rceil}^{\infty}\sum_{\iii\in D_{m,r}} \varphi^s(\iii)\exp(\sup_{\jjj\in[\iii]}\langle q,S_m\Phi(\jjj)-m\alpha\rangle)\\
&\leq\sum_{m=\lceil-\log\delta\rceil}^{\infty}e^{mP/2}.
\end{align*}
By letting $\delta \downarrow 0$, the upper bound above approaches to zero and hence, by letting $r \to \infty$, $\dimh(\pi E_\Phi(\alpha))\leq s$. The proof is finished as $s>s_0(\alpha)$ was arbitrary.
\end{proof}

\section{Birkhoff averages}\label{sec:domin}

We let $\mathbb{RP}^1$ denote the real projective line, which is the set of all lines through the origin in $\R^2$. We call a proper subset $\CC \subset \mathbb{RP}^1$ a \emph{cone} if it is a closed projective interval and a \emph{multicone} if it is a finite union of cones. Let $\mathsf{A} = (A_1,\ldots,A_N) \in GL_2(\R)^N$. We say that $\mathsf{A}$ is \emph{dominated} if there exists a multicone $\CC \subset \RP$ such that $\bigcup_{i=1}^NA_i\CC\subset\CC^o$. By \cite[Theorem B]{BochiGourmelon2009}, $\mathsf{A}$ is dominated if and only if there exist constants $C>0$ and $0<\tau<1$ such that
\begin{equation*}
  \frac{|\det(A_\iii)|}{\|A_\iii\|^2} \le C\tau^n
\end{equation*}
for all $\iii \in \Sigma_n$ and $n \in \N$. Furthermore, if $\mathsf{A}$ is dominated, then, by \cite[Lemma 2.4]{BaranyRams2017}, the mapping $V\colon\Sigma\to\RP$ defined by
\begin{equation}\label{eq:oseledetsspace}
  V(\iii)=\bigcap_{k=0}^{\infty}A_{\iii|_k}\CC
\end{equation}
is H\"older continuous.

\begin{proposition}\label{thm:BGprop}
  If $(A_1,\ldots,A_N) \in GL_2(\R)^N$ is dominated, then there exists $C>0$ such that
  $$
    \|A_\iii|V(\sigma^n\iii)\|\geq C\|A_\iii\|
  $$
  for all $\iii \in \Sigma_n$ and $n \in \N$. In particular, the function $\iii \mapsto \log\|A_{\iii|_1}|V(\sigma\iii)\|$ is H\"older continuous and
  $$
    |\log\|A_{\iii|_n}\|-S_n\log\|A_{\iii|_1}|V(\sigma\iii)\||\leq C
  $$
  for all $\iii \in \Sigma$ and $n \in \N$.
\end{proposition}

\begin{proof}
  This follows from Bochi and Morris \cite[Lemma~2.2]{BochiMorris15} and B\'ar\'any and Rams \cite[Lemma 2.4]{BaranyRams2017}.
\end{proof}

If $(A_1\ldots,A_N) \in GL_2(\R)^N$ is dominated, then we define $\Psi = (\Psi_1,\Psi_2) \colon \Sigma \to \R^2$ by setting
\begin{equation} \label{eq:Psi-def}
\begin{split}
  \Psi_1(\iii) &= -\log\|A_{\iii|_1}|V(\sigma\iii)\|, \\
  \Psi_2(\iii) &= -\log|\det(A_{\iii|_1})|+\log\|A_{\iii|_1}|V(\sigma\iii)\|,
\end{split}
\end{equation}
for all $\iii \in \Sigma$. Proposition \ref{thm:BGprop} implies that $\Psi$ is a H\"older continuous potential and
\begin{equation*}
  \int_\Sigma \Psi \dd\mu=\chi(\mu)
\end{equation*}
for all $\mu\in\MM_\sigma(\Sigma)$, where $\chi$ is defined in \eqref{eq:lyapunov-def}. We also define
\begin{equation*}
  \Psi^s(\iii) =
  \begin{cases}
    -s\Psi_1(\iii), &\text{if } 0 \le s < 1, \\
    -\Psi_1(\iii) - (s-1)\Psi_2(\iii), &\text{if } 1 \le s < 2, \\
    -s(\Psi_1(\iii)+\Psi_2(\iii))/2, &\text{if } 2 \le s < \infty,
  \end{cases}
\end{equation*}
for all $\iii \in \Sigma$.

Given a potential $\Phi \colon \Sigma\to\R^M$, we wish to find, for any $\alpha\in\PP(\Phi)^o$, a fully supported $n$-step Bernoulli measure $\nu$ with Birkhoff average $\int_\Sigma \Phi \dd\nu=\alpha$ and with large Lyapunov dimension. We will use Proposition \ref{thm:BGprop} to find such measures with Birkhoff average close to $\alpha$, but to find the exact match we will need the following technical lemma which is along the lines of Hopf's Lemma.

\begin{lemma} \label{lem:hopf}
  Let $f \colon \R^M\to (0,\infty)$ be a $C^1$ convex function and $g=\nabla f$. For $a>0$ let $Q=\{x\in\R^M : f(x)\leq a\}$. If $Q$ is bounded and nonempty such that $|g(x)|>b>0$ for all $x\in\partial Q$, then $B(0,b)\subset g(Q)$.
\end{lemma}

\begin{proof}
$Q$ is convex and $\partial Q$ is a smooth manifold. For any $x\in\partial Q$, $g(x)/|g(x)|$ is the outer normal to $\partial Q$ at $x$. As we can perturb $Q$ slightly to obtain a strictly convex set, we can find a (positively oriented) bijection $n \colon S^{M-1}\to \partial Q$ such that $n^{-1}(x)$ is an outer vector at every $x\in\partial Q$. The map $\ell=g\circ n/|g\circ n|$ maps $S^{M-1}$ into itself and satisfies $\langle y, \ell(y) \rangle\geq 0$, hence $\ell$ is homotopic to the identity (the homotopy is given by $h(y,\gamma)=(\gamma y + (1-\gamma)\ell(y))/|\gamma y + (1-\gamma)\ell(y)|$). This implies that the topological degree is $\deg(g,Q,0) = \deg(\ell,S^{M-1},0)=1$.

As $g(x)\geq b$ for all $x\in\partial Q$, for every $r\in \R^M, |r|<b$, $g$ is homotopic to $g-r$ rel 0 on $\partial Q$. Thus, $\deg(g,Q,r)=\deg(g-r,Q,0)=\deg(g,Q,0)=1$. Since this implies $B(0,b)\subset g(Q)$, we are done.
\end{proof}

\begin{proposition}\label{prop:approx2}
  Let $\mathsf{A} = (A_1,\ldots,A_N) \in GL_2(\R)^N$ be dominated and $\Phi\colon\Sigma\to\R^M$ be a  continuous potential. If $\alpha\in\PP(\Phi)^o$, then for every
  \begin{equation*}
    s < \sup\{ s \ge 0 : \inf_{q \in \R^M}P(\log\varphi^s +\langle q, \Phi - \alpha \rangle) \geq 0 \}
  \end{equation*}
  there exists a fully supported $n$-step Bernoulli measure $\nu$ such that $\diml(\nu)=\diml(\tilde{\nu})\ge s$ and $\int_\Sigma \Phi \dd\tilde{\nu}=\alpha$, where $\tilde{\nu}$ is defined in \eqref{nbern}.
\end{proposition}

\begin{proof}
Let $s < \sup\{ s \ge 0 : \inf_{q \in \R^M}P(\log\varphi^s +\langle q, \Phi - \alpha \rangle) \geq 0 \}$ and note that, by Proposition \ref{thm:BGprop}, we have
$$
  P(\log\varphi^s +\left\langle q, \Phi - \alpha \right\rangle)=P(\Psi^s+\left\langle q, \Phi - \alpha \right\rangle).$$
Since for any $q\in\R^m$ $s\to P(\Psi^s+\left\langle q, \Phi - \alpha \right\rangle)$ is strictly decreasing we have that
$$\inf\{P(\Psi^s+\langle q,\Phi-\alpha\rangle):q\in\R^{M}\}>0.$$
Since $\alpha\in\PP(\Phi)^o$ we can find $\eta>0$ such that for any $q\in\R^M$ with $|q|=1$ we can find an invariant measure $\mu$ such that $\int\Phi\dd\mu-\alpha=\eta q$. Therefore by the variational principle for any $q\in\R^m$
$$P(\Psi^s+\left\langle q, \Phi - \alpha \right\rangle)\geq h(\mu)+\int\log\Phi^s\dd\mu+\eta|q|$$
where $h(\mu)+\int\log\Phi^s\dd\mu$ is bounded uniformly below for all invariant measures.
Thus for
$$
\delta=\inf\{P(\Psi^s+\langle q,\Phi-\alpha\rangle):q\in\R^{M}\}>0
$$
we can choose $q_0>0$ so that
$$
P(\Psi^s+\langle q,\Phi-\alpha\rangle)\geq 3\delta
$$
whenever $|q| \geq q_0$. We fix $\eps_1,\eps_2>0$ such that
$$
\eps_1 q_0 + \eps_2<\delta/4.
$$
Since $\Phi$ and $\Psi^s$ are continuous, we can, by Lemma \ref{lem:variance}, choose $n\in\N$ such that
$$
\max_{\iii\in\Sigma_n}\max_{\jjj,\kkk\in [\iii]}\{|S_n\Phi(\jjj)-S_n\Phi(\kkk)|\}\leq n\eps_1
$$
and
$$
\max_{\iii\in\Sigma_n}\max_{\jjj,\kkk\in [\iii]}\{|S_n\Psi^s(\jjj)-S_n\Psi^s(\kkk)|\}\leq n\eps_2.
$$
Therefore, we can find functions $\Phi_n$ and $\Psi_n$ which are constant on $n$th level cylinders and where
$$
\norm{S_n\Phi-\Phi_n}_{\infty}\leq n\eps_1 \quad \text{and} \quad \norm{S_n\Psi^s-\Psi_n}_{\infty}\leq n\eps_2.
$$
We now work with the pressure for $\sigma^n$ which we denote by $P_n$. Note that we have $P_n(S_n \cdot)=nP(\cdot)$. Thus we have
$$
\inf\{P_n(S_n\Psi^s+\langle q,S_n\Phi-n\alpha\rangle):|q|\leq q_0\}=n\delta
$$
and
$$
\inf\{P_n(S_n\Psi^s+\langle q,S_n\Phi-n\alpha\rangle):|q|=q_0\}\geq 3n\delta.
$$
Since $\eps_1|q_0|+\eps_2<\delta/4$, we see that
\begin{equation}\label{eq:min}
  \min\{ P_n(\Psi_n+\langle q,\Phi_n-n\alpha\rangle):|q|\leq q_0\} \in \biggl[\frac{3n\delta}{4}, \frac{5n\delta}{4}\biggr]
\end{equation}
and
$$
\min\{ P_n(\Psi_n+\langle q,\Phi_n-n\alpha\rangle):|q|=q_0\}\geq \frac{11 n\delta}{4}.
$$
Since $\Phi_n$ and $\Psi_n$ are locally constant, and therefore H\"{o}lder continuous, the function $q \mapsto P_n(\Psi_n+\langle q,\Phi_n-n\alpha\rangle)$ is analytic and convex. Moreover, for any $q_*\in \R^M$ we have
$$
\nabla\big|_{q=q_*} P_n(\Psi_n+\langle q,\Phi_n-n\alpha\rangle)=\int_\Sigma (\Phi_n-n\alpha)\dd\mu_{q_*},
$$
where $\mu_{q_*}$ is the equilibrium state for $\Psi_n+\langle q_*,\Phi_n-n\alpha\rangle$. Note that the set
$$
Q=\{q:P_n(\Psi_n+\langle q,\Phi_n-n\alpha\rangle)\leq 2n\delta\} \subset B(0,q_0)
$$
is convex. By convexity and \eqref{eq:min}, we get
\begin{align*}
  |\nabla P_n&(\Psi_n+\langle q,\Phi_n-n\alpha\rangle)| \\ &\geq \frac{P_n(\Psi_n+\langle q,\Phi_n-n\alpha\rangle)-P_n(\Psi_n+\langle\tilde{q},\Phi_n-n\alpha\rangle)}{|q-\tilde{q}|} \geq\frac{3n\delta}{8q_0}
\end{align*}
for all $q\in\partial Q$.

Define $f_1,f_2\colon Q\to\R^M$ by setting
$$
f_1(q)=\int_\Sigma (\Phi_n-n\alpha) \dd\mu_q=\nabla P_n(\Psi_n+\langle q,\Phi_n-n\alpha\rangle)
$$
and
$$
f_2(q)=\int_\Sigma (S_n\Phi-n\alpha) \dd\mu_q.
$$

By Lemma \ref{lem:hopf}, $f_1(Q)\supset B(0,\frac{3n\delta}{8q_0})$. Since
$$
\norm{f_1-f_2}_{\infty}\leq\eps_1n<\frac{3n\delta}{8q_0},
$$
we have
$$
|(1-t)f_1(q)+tf_2(q)|\geq \frac{3\delta n}{8q_0}-t|f_1(q)-f_2(q)|\geq\frac{3\delta n}{8q_0}-\eps_1n>0
$$
for all $t\in[0,1]$ and $q\in\partial Q$. It follows that $f_1|_{\partial Q}$ and $f_2|_{\partial Q}$ are homotopic on $\R^{M}\setminus\{0\}$ and hence $0\in f_2(Q)$. This means that there exists $q_1\in Q$ such that $\int_\Sigma S_n\Phi\dd\mu_{q_1}=n\alpha$. Now, by \eqref{eq:min},
\begin{align*}
  \frac{3n\delta}{4} &\leq P_n(\Psi_n+\langle q_1,\Phi_n-n\alpha\rangle)\\
  &=h(\mu_{q_1},\sigma^n)+\int_\Sigma \Psi_n+\langle q_1,\Phi_n-n\alpha\rangle\dd\mu_{q_1}.
\end{align*}
So
$$
  h(\mu_{q_1},\sigma^n)+\frac{n\delta}{4}+\int_\Sigma S_n\Psi^s\dd\mu_{q_1}\geq \frac{3n\delta}{4}
$$
and
$$
  h(\mu_{q_1},\sigma^n)+\int_\Sigma S_n\Psi^s\dd\mu_{q_1}\geq 0.
$$
If $0 \leq s < 1$, then we have
$$
\diml(\mu_{q_1})=\frac{h(\mu_{q_1},\sigma^n)}{\chi_1(\mu_{q_1},\sigma^n)}=\frac{h(\mu_{q_1},\sigma^n)}{\int_\Sigma S_n\Psi_1\dd\mu_{q_1}}\geq s.
$$
Alternatively, if $1 \leq s < 2$, then
$$h(\mu_{q_1},\sigma^n)+\chi_1(\mu_{q_1},\sigma^n)+(s-1)\chi_2(\mu_{q_1},\sigma^n)\geq 0$$
and rearranging gives that
$$
\diml(\mu_{q_1})=1+\frac{h(\mu_{q_1},\sigma^n)-\chi_1(\mu_{q_1},\sigma^n)}{\chi_2(\mu_{q_1},\sigma^n)}\geq s.
$$
The case $s \geq 2$ is left to the reader. Since $\mu_{q_1}$ is an equilibrium state for $\Psi_n+\langle q_1,\Phi_n-n\alpha\rangle$, which is constant on $n$th level cylinders, it is a $\sigma^n$-invariant Bernoulli measure. Taking $\nu=\mu_{q_1}$ finishes the proof.
\end{proof}

We remind the reader that the (lower) Hausdorff dimension of the measure $\mu$ on $\R^2$ is defined by
$$
\dimh(\mu)=\inf\{\dimh(A) : \mu(A)>0\}.
$$
In order to provide the lower bounds in the main theorems, we find invariant measures with prescribed integrals and Lyapunov exponents, for which we can calculate the Hausdorff dimension. The following theorem guarantees that  $n$-step Bernoulli measures can be used for this purpose.

\begin{theorem}\label{thm:dimension}
Let $(A_1+v_1,\ldots,A_N+v_N)$ be an affine IFS on $\R^2$ satisfying the SOSC. If $(A_1,\ldots,A_N) \in GL_2(\R)^N$ is strongly irreducible such that the generated subgroup of the normalized matrices is non-compact, then
$$
  \dimh(\pi_*\mu) = \diml(\mu)
$$
for all Bernoulli measures $\mu$ on $\Sigma$.
\end{theorem}

\begin{proof}
  This is by B\'ar\'any, Hochman, and Rapaport \cite[Theorem~1.2]{bhr2017selfaffine}.
\end{proof}

\begin{remark} \label{rem:hochman-rapaport}
  In a very recent preprint of Hochman and Rapaport \cite{hr2019selfaffine} the above result has been generalized to a case which allows severe overlapping. The strong open set condition has been replaced by the assumption that the defining affine maps do not share a fixed point and are exponentially separated; see \cite[Theorem 1.1]{hr2019selfaffine}. Relying on this, instead of Theorem \ref{thm:dimension}, would improve our main results accordingly.  
\end{remark}

We are now able to prove Theorem \ref{thm:mainbirkhoffint} for dominated systems.

\begin{theorem}\label{thm:domibirkhoff}
Let $(A_1+v_1,\ldots,A_N+v_N)$ be an affine IFS on $\R^2$ satisfying the SOSC and $\Phi\colon\Sigma\to\R^M$ be a continuous potential. If $(A_1,\ldots,A_N) \in GL_2(\R)^N$ is strongly irreducible and dominated such that the generated subgroup of the normalized matrices is non-compact, then
\begin{align*}
  \dimh(\pi E_{\Phi}(\alpha)) &= \sup\{\diml(\mu) : \mu\in\MM_\sigma(\Sigma)\text{ and }\int_\Sigma\Phi \dd\mu=\alpha\} \\
  &=\sup\{ s \ge 0 : \inf_{q \in \R^M}P(\log \varphi^s +\left\langle q, \Phi - \alpha \right\rangle) \geq 0 \}
\end{align*}
for all $\alpha\in\PP(\Phi)^o$.
\end{theorem}

\begin{proof}
It follows from Propositions~\ref{thm:birkhoffupper} and \ref{prop:approx2} that
\begin{align*}
  \dimh(\pi E_\Phi(\alpha)) &\leq \sup\{ s \ge 0 : \inf_{q \in \R^M}P(\log \varphi^s +\left\langle q, \Phi - \alpha \right\rangle) \geq 0 \} \\
  &\leq \sup\{\diml(\tilde\nu) :\nu \text{ is fully supported $n$-step} \\
  &\qquad\qquad\qquad\quad\;\;\,\text{Bernoulli and } \int_\Sigma\Phi \dd\tilde\nu=\alpha\} \\
  &\leq \sup\{\diml(\mu) : \mu\in\MM_\sigma(\Sigma)\text{ and }\int_\Sigma\Phi \dd\mu=\alpha\},
\end{align*}
where $\tilde\nu$ is defined in \eqref{nbern}. Let $\mu \in \MM_\sigma(\Sigma)$ be such that $\int_\Sigma \Phi \dd\mu = \alpha$. By the variational principle (see \cite{CaoFengHuang, Kaenmaki2004}), if $s < \diml(\mu)$, then
\begin{equation*}
  P(\log \varphi^s +\left\langle q, \Phi - \alpha \right\rangle) \geq h(\mu) + \lim_{n \to \infty} \tfrac1n \int_\Sigma \log\fii^s(A_{\iii|_n}) \dd\mu(\iii) > 0
\end{equation*}
for all $q \in \R^M$. Therefore, $s \leq \sup\{ s \ge 0 : \inf_{q \in \R^M}P(\log \varphi^s +\left\langle q, \Phi - \alpha \right\rangle) \geq 0 \}$ and, consequently,
\begin{align*}
  \sup\{\diml(\mu) :\;&\mu\in\MM_\sigma(\Sigma)\text{ and }\int_\Sigma\Phi \dd\mu=\alpha\} \\
  &\leq \sup\{ s \ge 0 : \inf_{q \in \R^M}P(\log \varphi^s +\left\langle q, \Phi - \alpha \right\rangle) \geq 0 \}.
\end{align*}
Finally, let $\nu$ be a fully supported $n$-step Bernoulli measure so that $\int_\Sigma\Phi \dd\tilde\nu=\alpha$. Then clearly $\tilde\nu(E_\Phi(\alpha))=1$ and, by Theorem \ref{thm:dimension},
$$
  \diml(\tilde{\nu})=\dimh(\pi_*\tilde{\nu}) \leq \dimh(\pi E_\Phi(\alpha))
$$
finishing the proof.
\end{proof}

\section{Dominated subsystems}\label{sec:subsys}

We begin the section with the proof of a key lemma in order to construct dominated subsystems.

\begin{lemma}\label{lem:domin}
Let $A_1,A_2,A_3\in GL_2(\R)$ such that there exist cones $\BB_1, \BB_2,\BB_3$ and $\CC_1, \CC_2, \CC_3$ such that
\begin{enumerate}
  \item $\BB_1\cap\BB_2=\emptyset$, $\CC_1\cap\CC_2=\emptyset$, $\CC_1\cap\BB_1=\emptyset$, and $\CC_2\cap\BB_2=\emptyset$,
  \item there exist $i_1,j_1\in\{1,2\}$ such that $\CC_3\subset\BB_{i_1}^o$ and $\BB_3\subset\CC_{j_1}^o$,
  \item $A_i(\overline{\RP\setminus\BB_i})\subset\CC_i^o$ for every $i\in\{1,2,3\}$.
\end{enumerate}
We then have that for every $A\in GL_2(\R)$ which is hyperbolic with $u(A)\in \CC_i^o$ and $s(A)\in \BB_j^o$ there exist $i,j\in\{1,2\}$ such that $A_iA_3AA_j^2(\CC_1\cup\CC_2)\subset (\CC_1\cup\CC_2)^o$. For every other $A\in GL_2(\R)$ there exist $i,j\in\{1,2\}$ such that $A_i^2AA_j^2(\CC_1\cup\CC_2)\subset (\CC_1\cup\CC_2)^o$.
\end{lemma}

\begin{proof}
First, let us make a couple of remarks. It is easy to see that $A_i(\CC_1\cup\CC_2)\subset \CC_i^o$ for $i\in\{1,2\}$.  This means that $A_{3-i_1}\CC_3\subset\CC_{3-i_1}^o$. Finally, we note that if a cone $\CC$ satisfies $\CC\cap(\CC_1\cup\CC_2)=\emptyset$, then $A_3\CC\subset\CC_3^o$.

Fix $A\in GL_2(\R)$. We see that there are two possible cases:
\begin{enumerate}
\item\label{eq:case1} there exist $i,j\in\{1,2\}$ such that $A\CC_j\cap\BB_i^o=\emptyset$,
\item\label{eq:case2} for every $i,j\in\{1,2\}$ we have $A\CC_i\cap\BB_j^o\neq\emptyset$.
\end{enumerate}
In the case~\eqref{eq:case1}, since $A\CC_j\cap\BB_i^o=\emptyset$, we have $A_i(A\CC_j)\subset A_i(\overline{\RP\setminus\BB_i})\subset\CC_i^o$. Thus,
$$
A_i^2AA_j^2(\CC_1\cup\CC_2)\subset A_i^2A\CC_j^o\subset \CC_i^o\subset(\CC_1\cup\CC_2)^o.
$$
On the other hand, if the case~\eqref{eq:case2} holds, then $A(\overline{\RP\setminus(\CC_1\cup\CC_2)})\subset\BB_1^o\cup\BB_2^o$. Since $\BB_1$ and $\BB_2$ are disjoint intervals on $\RP$, one of the connected components of $(\overline{\RP\setminus(\CC_1\cup\CC_2)})$ is contained in $\BB_1^o$, and the other one is contained in $\BB_2^o$. In particular, there are $k,k'\in\{1,2\}$ so that $A\BB_1\subset\BB_k^o$ and $A\BB_2\subset\BB_{k'}^o$. Now, if $k\neq k'$ then $A^2\BB_1\subset\BB_1^o$ and $A^2\BB_2\subset\BB_2^o$, which is a contradiction, since it would imply that $A^2$ has two different stable eigenspaces. Similar argument can be applied for the cones $\CC_1$ and $\CC_2$, and the inverse matrix $A^{-1}$.

Thus, there exist unique $i,j\in\{1,2\}$ such that $A\BB_i\subset\BB_i^o$ and $A^{-1}\CC_j\subset\CC_j^o$. Thus, in particular, $A$ is a hyperbolic matrix with stable and unstable eigenspaces $s(A)\in\BB_i^o$ and $u(A)\in\CC_j^o$. Moreover, $A\CC_{3-j}\cap(\CC_1\cup\CC_2)=\emptyset$; see Figure~\ref{fig:illustration}. Thus,
$$
A_{3-i_1}A_3AA_{3-j}^2(\CC_1\cup\CC_2)\subset A_{3-i_1}A_3A\CC_{3-j}^o\subset A_{3-i_1}\CC_3^o\subset\CC_{3-i_1}^o\subset(\CC_1\cup\CC_2)^o
$$
and the proof is finished.
\end{proof}

\begin{center}
\begin{figure}[t]

\begin{tikzpicture}[scale=0.9]
  % Circle on the left
  \draw (0,0) circle [radius=3];
  \draw[|-|] [thick,domain=-0.7:0.7,smooth,variable=\x] plot ({\x},{sqrt{(9.5-(\x)*(\x))}});
  \node[align=center] at (0.1,3.42) {$\mathcal{C}_i$};
  \draw[|-|] [thick,domain=1.5:2.5,smooth,variable=\x] plot ({\x},{sqrt{(9.5-(\x)*(\x))}});
  \node[align=center] at (2.45,2.6) {$\mathcal{C}_{3-i}$};
  \draw[|-|] [thick,domain=-1.9:-1.55,smooth,variable=\x] plot ({\x},{-sqrt{(8.5135-(\x)*(\x))}});
  \node[align=center] at (-1.1,-2.4) {$\mathcal{C}_{3}$};
  \draw[|-|] [thick,domain=0:1.5,smooth,variable=\x] plot ({\x},{-sqrt{(9.5-(\x)*(\x))}});
  \node[align=center] at (-2.2,-2.7) {$\mathcal{B}_j$};
  \draw[|-|] [thick,domain=-2.5:-1.5,smooth,variable=\x] plot ({\x},{-sqrt{(9.5-(\x)*(\x))}});
  \node[align=center] at (1,-3.4) {$\mathcal{B}_{3-j}$};
  \draw[|-|] [thick,domain=2:2.3,smooth,variable=\x] plot ({\x},{sqrt{(8.5135-(\x)*(\x))}});
  \node[align=center] at (2.3,1.4) {$\mathcal{B}_{3}$};

  % Arrows in the left ball
  \draw [decoration={markings,mark=at position 0.7 with
    {\arrow[scale=2.5,>=stealth]{>}}},postaction={decorate}] (1.95,1.8) .. controls (0,0) .. (-1.6,-2.2);
  \draw [decoration={markings,mark=at position 0.7 with
    {\arrow[scale=2.5,>=stealth]{>}}},postaction={decorate}] (-2.1,-1.9) .. controls (-0.4,0.4) .. (1.6,2.2);
  \draw [decoration={markings,mark=at position 0.7 with
    {\arrow[scale=2.5,>=stealth]{>}}},postaction={decorate}] (0.6,-2.75) .. controls (-0.5,0.4) .. (0,2.8);

  % Circle on the right
  \draw (7,0) circle [radius=3];
  \draw[|-|] [thick,domain=6.3:7.7,smooth,variable=\x] plot ({\x},{sqrt{(9.5-(\x-7)*(\x-7))}});
  \node[align=center] at (7.1,3.42) {$\mathcal{C}_i$};
  \draw[|-|] [thick,domain=8.5:9.5,smooth,variable=\x] plot ({\x},{sqrt{(9.5-(\x-7)*(\x-7))}});
  \node[align=center] at (9.45,2.6) {$\mathcal{C}_{3-i}$};
  \draw[|-|] [thick,domain=7:8.5,smooth,variable=\x] plot ({\x},{-sqrt{(9.5-(\x-7)*(\x-7))}});
  \node[align=center] at (4.8,-2.7) {$\mathcal{B}_i$};
  \draw[|-|] [thick,domain=4.5:5.5,smooth,variable=\x] plot ({\x},{-sqrt{(9.5-(\x-7)*(\x-7))}});
  \node[align=center] at (8,-3.4) {$\mathcal{B}_{3-i}$};
  \draw[|-|] [thick,domain=5.4:7.2,smooth,variable=\x] plot ({\x},{-sqrt{(8.5135-(\x-7)*(\x-7))}});
  \node[align=center] at (6.55,-2.5) {$A\mathcal{C}_{3-j}$};
  \draw[|-|] [thick,domain=-65:220,smooth,variable=\x] plot ({7+2.91779026*cos(\x)}, {2.91779026*sin(\x)});
  \node[align=center] at (5,1.4) {$A\mathcal{C}_{j}$};

  \end{tikzpicture}
  \caption{The left-hand side image illustrates the assumptions in Lemma \ref{lem:domin} and the image on the right depicts one of the possible situations in the case (2) of the proof.}
  \label{fig:illustration}
\end{figure}
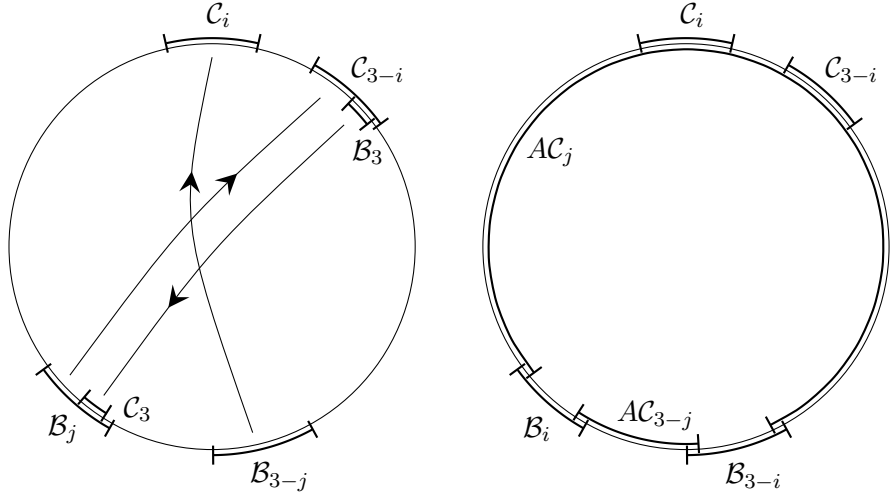
\end{center}

\begin{lemma}\label{lem:consdomi}
Let $(A_1,\ldots,A_N) \in GL_2(\R)^N$ be strongly irreducible such that the generated subgroup of the normalized matrices is non-compact. Then there exist $K\in\N$ and multicones $\BB$ and $\CC$ such that $\BB\subset\CC^o$, and for every $\iii\in\Sigma_*$ there exist $\jjj_1,\jjj_2\in\Sigma_K$ such that
$$
A_{\jjj_1}A_{\iii}A_{\jjj_2}\CC\subset\BB^o.
$$
\end{lemma}

\begin{proof}
Since the tuple is strongly irreducible and the generated subgroup of the normalized matrices is non-compact, there exist $\kkk_1,\kkk_2\in\Sigma_*$ such that $A_{\kkk_1}$ and $A_{\kkk_2}$ are hyperbolic and the spaces $s(A_{\kkk_1}), s(A_{\kkk_2}), u(A_{\kkk_1}), u(A_{\kkk_2})$ are all different. By taking powers, one can choose $\kkk_1$ and $\kkk_2$ so that $|\kkk_1|=|\kkk_2|$. Let $r>0$ be so small such that
\begin{equation}\label{eq:disjoint}
\begin{split}
B(s(A_{\kkk_1}),r)\cap B(s(A_{\kkk_2}),r)&=B(u(A_{\kkk_1}),r)\cap B(u(A_{\kkk_2}),r)\\
&=B(s(A_{\kkk_i}),r)\cap B(u(A_{\kkk_j}),r)=\emptyset,
\end{split}
\end{equation}
where $B(x,r)$ denotes the closed ball centered at $x$ and with radius $r$. Thus, there exists $L=L(r)\geq1$ such that
$$
(A_{\kkk_i})^L(\overline{\RP\setminus B(u(A_{\kkk_i}),r)})\subset B(s(A_{\kkk_i}),r)^o
$$
for $i\in\{1,2\}$. We distinguish two cases:
\begin{enumerate}
  \item\label{eq:case1b} there exists $r>0$ such that $u(A_{\iii})\notin B(s(A_{\kkk_i}),r)$ or $s(A_{\iii})\notin B(u(A_{\kkk_i}),r)$ for all $\iii\in\Sigma_*$ and $i\in\{1,2\}$,
  \item\label{eq:case2b} for every $r>0$ there exists $\iii\in\Sigma_*$ such that $u(A_{\iii})\in B(s(A_{\kkk_i}),r)$ and $s(A_{\iii})\in B(u(A_{\kkk_j}),r)$ for some $i,j\in\{1,2\}$.
\end{enumerate}
In the case~\eqref{eq:case1b}, by Lemma~\ref{lem:domin}, we see that for every $\iii\in\Sigma_*$ there exist $i,j\in\{1,2\}$ so that
\begin{align*}
    (A_{\kkk_i})^{2L}A_{\iii}(A_{\kkk_j})^{2L}(B(s(&A_{\kkk_1}),r)\cup B(s(A_{\kkk_2}),r)) \\ &\subset (B(s(A_{\kkk_1}),r)\cup B(s(A_{\kkk_2}),r))^o.
\end{align*}
If the case~\eqref{eq:case2b} holds, then fix $r>0$ so that \eqref{eq:disjoint} holds. Let $\kkk_3$ be such that $u(A_{\kkk_3})\in B(s(A_{\kkk_i}),r)$ and $s(A_{\kkk_3})\in B(u(A_{\kkk_j}),r)$. By choosing $\rho>0$ sufficently small, we have $B(u(A_{\kkk_3}),\rho)\subset B(s(A_{\kkk_i}),r)$ and $B(s(A_{\kkk_3}),\rho)\subset B(u(A_{\kkk_j}),r)$. Therefore, by choosing $M\in\N$ sufficiently large, we have
$$
(A_{\kkk_3})^M(\overline{\RP\setminus B(u(A_{\kkk_i}),\rho)})\subset B(s(A_{\kkk_i}),\rho)^o.
$$
By taking powers, we can assume that $M|\kkk_3|=|\kkk_1|=|\kkk_2|$.
Thus, the statement of the lemma again follows by applying Lemma~\ref{lem:domin}, with $K=2L\max|\kkk_1|$  and $\CC=B(s(A_{\kkk_1}),r)\cup B(s(A_{\kkk_2}),r)$.
\end{proof}

By Lemma~\ref{lem:consdomi}, there exists $K\geq1$ such that for every $n> 2K$ and $\iii\in\Sigma_{n-2K}$ there exist $\jjj_1=\jjj_1(\iii)$ and $\jjj_2=\jjj_2(\iii)$ such that the tuple
\begin{equation}\label{eq:dominsub}
  (A_\kkk)_{\kkk \in \Sigma_n^{\DD}}, \qquad \text{where }  \Sigma_n^{\DD} = \{\jjj_1(\iii)\iii\jjj_2(\iii) : \iii\in\Sigma_{n-2K}\},
\end{equation}
is dominated and strongly irreducible. Note that $\Sigma_n^\DD\subset\Sigma_{n}$ for all $n > 2K$.
If $n>2K$ and $\varphi \colon \Sigma\to\R$ is a subadditive potential, then we define a pressure on the dominated subsystem by setting
$$
P_{\DD,n}(\log\varphi)=\lim_{k\to\infty}\tfrac{1}{k}\log\sum_{\iii_1,\ldots,\iii_k\in\Sigma_n^\DD}\varphi(\iii_1\cdots\iii_k).
$$
Note that the condition given by Lemma \ref{lem:consdomi} is stronger than the usual domination, see Avila, Bochi, and Yoccoz \cite{AvilaBochiYoccoz2010} and Bochi and Gourmelon \cite{BochiGourmelon2009}. In particular, we have the following uniform bound.

\begin{lemma} \label{lem:dominbound}
Let multicones $\BB$ and $\CC$ be such that $\BB\subset\CC^o$. Then there exists $Z>0$ such that for every pair of matrices $A$ and $B$ satisfying $A\CC \cup B\CC\subset \BB^o$ we have
\[
\norm{AB} \geq e^{-Z} \norm{A}\norm{B}.
\]
\end{lemma}

\begin{proof}
  This follows easily from Bochi and Morris \cite[Lemma~2.2]{BochiMorris15}.
\end{proof}

The lemma guarantees that in every subsemigroup of $\bigcup_{n \in \N} \Sigma_n^\DD$ the norm (and hence also $\varphi^s$ for all $s\in[0,\infty)$) is almost multiplicative up to a uniformly chosen constant.

\begin{lemma}\label{lem:pressapprox}
Let $\Phi\colon\Sigma\to\R^M$ be a continuous potential. Then
$$
\lim_{n\to\infty}\tfrac{1}{n}P_{\DD,n}(\log\varphi^s +\left\langle q, S_n\Phi - \alpha \right\rangle) = P(\log\varphi^s +\left\langle q, \Phi - \alpha \right\rangle),
$$
uniformly for all $(q,\alpha,s)$ on any compact subset of $\R^M\times\PP(\Phi)\times[0,\infty)$.
\end{lemma}

\begin{proof}
Since
\begin{align*}
\tfrac{1}{n}P_{\DD,n}(\log\varphi^s +\langle q, S_n\Phi - n\alpha \rangle) &\leq \tfrac{1}{n}P_n(\log\varphi^s +\langle q, S_n\Phi - n\alpha \rangle) \\
&= P(\log\varphi^s +\langle q, \Phi - \alpha \rangle)
\end{align*}
for all $n\in\N$, the upper bound follows immediately. To show the lower bound, note that, by Lemma \ref{lem:dominbound}, we have
\begin{equation*}
  \varphi^s(A_{\iii\jjj})\geq e^{-Zs}\varphi^s(A_\iii)\varphi^s(A_\jjj)
\end{equation*}
for all $\iii, \jjj \in \bigcup_{k=1}^\infty(\Sigma_n^\DD)^k$ and $n\in\N$. For simplicity, let us denote the $k$th Birkhoff sum with respect to $\sigma^n$ by $S_k^{(n)}\Phi(\iii)=\sum_{\ell=0}^{k-1}\Phi(\sigma^{\ell n}\iii)$ . Combining this with Lemma~\ref{lem:variance} gives
\begin{align*}
\tfrac1n &P_{\DD,n}(\log\varphi^s+\langle q, S_n\Phi - n\alpha \rangle)\\
&=\lim_{k\to\infty}\tfrac{1}{nk}\log\sum_{\iii_1,\ldots,\iii_k\in\Sigma_{n}^\DD}\varphi^s(\iii_1\cdots\iii_k)\exp\biggl(\sup_{\jjj\in[\iii_1\cdots\iii_k]}\langle q,S_k^{(n)}(S_n\Phi)-nk\alpha\rangle\biggr)\\
&\geq\lim_{k\to\infty}\tfrac{1}{nk}\log\sum_{\iii_1,\ldots,\iii_k\in\Sigma_{n}^\DD}e^{-Z(k-1)s}\varphi^s(\iii_1)\cdots\varphi^s(\iii_k)\\
&\qquad\cdot\exp\biggl(\sup_{\jjj\in[\iii_1\cdots\iii_k]}\langle q,S_k^{(n)}(S_n\Phi)-nk\alpha\rangle\biggr)\\
&\geq\lim_{k\to\infty}\tfrac{1}{nk}\log\sum_{\iii_1,\ldots,\iii_k\in\Sigma_{n}^\DD}e^{-Z(k-1)s}\varphi^s(\iii_1)\cdots\varphi^s(\iii_k)\\
&\qquad\cdot\exp\biggl(\sum_{\ell=1}^{k}\sup_{\jjj\in[\iii_\ell]}\langle q,(S_n\Phi)-n\alpha\rangle+k|q|\sum_{i=0}^{n-1}\Var_i(\Phi)\biggr)\\
&=\tfrac{1}{n}\log\sum_{\iii\in\Sigma_{n}^\DD}\varphi^s(\iii)\exp\biggl(\sup_{\jjj\in[\iii]}\langle q,(S_n\Phi)-n\alpha\rangle\biggr)+\frac{-Zs+|q|\sum_{i=0}^{n-1}\Var_i(\Phi)}{n}\\
&\geq\frac{n-2K}{n}\cdot\frac{1}{n-2K}\log\sum_{\iii\in\Sigma_{n-2K}}\varphi^s(\iii)\\
&\qquad\cdot\exp\biggl(\sup_{\jjj\in[\iii]}\langle q,(S_{n-2K}\Phi)-(n-2K)\alpha\rangle\biggr)+\frac{-Zs+|q|\sum_{i=0}^{n-1}\Var_i(\Phi)}{n}.
\end{align*}
The statement follows by taking $n\to\infty$.
\end{proof}

\begin{proposition}\label{cont}
For any continuous $\Phi\colon\Sigma\to\R^M$ and $\alpha\in\mathcal{P}(\Phi)^o$ the function
$$
  \alpha\mapsto\sup\{s\geq 0:\inf_{q\in\R^m} P(\log\varphi^s+\langle q,\Phi-\alpha\rangle)\geq 0\}
$$
is continuous.
\end{proposition}

\begin{proof}
Let $t$ satisfy $P(\log\varphi^t)=0$. In this case for any $s> t$
$$\inf_{q\in\R^m} P(\log\varphi^s+\langle q,\Phi-\alpha\rangle)\leq P(\log\varphi^s)<0$$
and so
$$
\sup\{s\geq 0:\inf_{q\in\R^m} P(\log\varphi^s+\langle q,\Phi-\alpha\rangle)\geq 0\}\leq t.
$$
By the variational principle (see \cite{CaoFengHuang, Kaenmaki2004}), we have
\begin{equation}\label{eq:varpri}
\begin{split}
P(\log\varphi^s+\langle q,\Phi-\alpha\rangle)&\geq h(\mu)+\lim_{n\to\infty}\tfrac1n \int_\Sigma \log\varphi^s(A_{\iii|_n})\dd\mu(\iii)\\
&\qquad+\Bigl\langle q,\int_\Sigma\Phi \dd\mu-\alpha\Bigr\rangle\\
&\geq -Ct+\Bigl\langle q,\int_\Sigma\Phi \dd\mu-\alpha\Bigr\rangle
\end{split}
\end{equation}
for all $\mu\in\MM_\sigma(\Sigma)$ and $s \in [0,t]$, where $C=\max\{\log\norm{A_i^{-1}} : i\in \{1,\ldots,N\}\}$. Since $\alpha\in\mathcal{P}(\Phi)^o$, there exists $\delta>0$ such that
$$
\{\beta \in \R^M : |\alpha_i-\beta_i| \leq \delta \text{ for all } i \in \{1,\ldots,M\}\} \subset \PP(\Phi)^o.
$$
Hence, for each $i \in \{1,\ldots,M\}$ there exist $\mu_1,\mu_2\in\MM_\sigma(\Sigma)$ such that $\int_\Sigma\Phi \dd\mu_1=(\alpha_1,\ldots,\alpha_i+\delta,\ldots,\alpha_M)$ and $\int_\Sigma\Phi \dd\mu_2=(\alpha_1,\ldots,\alpha_i-\delta,\ldots,\alpha_M)$ and therefore, by \eqref{eq:varpri},
$$
P(\log\varphi^s+\langle q,\Phi-\alpha\rangle)\geq |q_i|\delta-tC.
$$
Thus, for every $s \in [0,t]$ we have
$$
P(\log\varphi^s+\langle q,\Phi-\alpha\rangle)>0
$$
unless $q\in [-tC/\delta,tC/\delta]^M$.

Since, by Lemma \ref{lem:pressapprox},
$$
(q,\alpha,s)\mapsto P(\log \varphi^s +\left\langle q, \Phi - \alpha \right\rangle)
$$
is uniformly continuous on $[-tC/\delta,tC/\delta]^M\times B(\alpha,\delta) \times [0,t]$. Therefore, for any $\eta>0$, we can choose $0<\eps<\delta$ such that, for any $s \in [0,t]$,
$$
|\inf_{q\in\R^m}\{P(\log\varphi^s+\langle q,\Phi-\beta\rangle)\}-\inf_{q\in\R^m}\{P(\log\varphi^s+\langle q,\Phi-\alpha\rangle)\}|\leq\eta
$$
for all $\beta\in B(\alpha,\eps)$. Notice also that for $s_1>s_2$ we have
\begin{align*}
(s_2-s_1)C_1&\leq P(\log\varphi^{s_1}+\langle q,\Phi-\beta\rangle)-P(\log\varphi^{s_2}+\langle q,\Phi-\beta\rangle)\\ &\leq (s_2-s_1)C,
\end{align*}
where $C_1=\min\{-\log\norm{A_i}:i \in \{1,\ldots,N\}\}$. This completes the proof.
\end{proof}

We are now ready to complete the proof of Theorem \ref{thm:mainbirkhoffint}.

\begin{proof}[Proof of Theorem~\ref{thm:mainbirkhoffint}]
Note that, for any $\jjj\in\Sigma$, we have $\lim_{n\to\infty}\tfrac{1}{n}S_n\Phi(\jjj)=\alpha$ if and only if $\lim_{k\to\infty}\tfrac{1}{nk}S_k^{(n)}S_n\Phi(\jjj)=\alpha$ for any $n\in\N$. Indeed, for any $n\in\N$ and $m\in\N$ with $nk\leq m<n(k+1)$, we have
\begin{align*}
  \bigl|\tfrac{1}{m}S_m\Phi(\jjj)-\tfrac{1}{nk}S_{nk}\Phi(\jjj)\bigr|&=\biggl|\tfrac{1}{m}S_m\Phi(\jjj)-\biggl(1+\frac{m-nk}{nk}\biggr)\tfrac{1}{m}S_{nk}\Phi(\jjj)\biggr|\\
  &\leq\biggl|\tfrac{1}{m}\sum_{\ell=nk+1}^m\Phi(\sigma^\ell\jjj)\biggr|+\tfrac{1}{k}\bigl|\tfrac{1}{nk}S_{nk}\Phi(\jjj)\bigr|\\
  &\leq\frac{2\sup|\Phi|}{k}\to 0
\end{align*}
as $k\to\infty$. Hence,
\begin{equation*}
  \pi E_{\Phi}(\alpha)=\pi E_{S_n\Phi}(n\alpha)\supset\pi E_{S_n\Phi}^\DD(n\alpha),
\end{equation*}
where
$$
E_{S_n\Phi}^\DD(n\alpha)=\{\jjj\in\left(\Sigma_n^\DD\right)^\N:\lim_{k\to\infty}\tfrac{1}{k}S_k^{(n)}S_n\Phi(\jjj)=n\alpha\}.
$$
Write
\begin{align*}
  s_0(\alpha) &= \sup\{ s \ge 0 : \inf_{q \in \R^M}P(\log\varphi^s +\langle q, \Phi - \alpha \rangle) \geq 0 \}, \\
  s_n(\alpha) &= \sup\{ s \ge 0 : \inf_{q \in \R^M}P_{\DD,n}(\log\varphi^s +\langle q, S_n\Phi - n\alpha \rangle) \geq 0 \},
\end{align*}
and note that, by Lemma \ref{lem:pressapprox},
\begin{equation*}
  \lim_{n\to\infty} s_n(\alpha) = s_0(\alpha).
\end{equation*}
By Theorem \ref{thm:domibirkhoff} and Proposition~\ref{thm:birkhoffupper}, we have
\begin{equation*}
  s_n(\alpha) = \dimh(\pi E_{S_n\Phi}^\DD(n\alpha)) \leq \dimh(\pi E_{\Phi}(\alpha)) \leq s_0(\alpha).
\end{equation*}
Therefore, by letting $n \to \infty$, we see that
\begin{equation*}
  \dimh(\pi E_{\Phi}(\alpha)) = s_0(\alpha).
\end{equation*}

On the other hand, let $\mu \in \MM_\sigma(\Sigma)$ be such that $\int_\Sigma \Phi \dd\mu = \alpha$. By the variational principle (see \cite{CaoFengHuang, Kaenmaki2004}), if $s < \diml(\mu)$, then
\begin{equation*}
  P(\log\fii^s + \langle q,\Phi-\alpha \rangle) \geq h(\mu) + \lim_{n\to\infty}\tfrac1n\int_\Sigma\log\fii^s(A_{\iii|_n})\dd\mu(\iii) > 0
\end{equation*}
for all $q \in \R^M$. Therefore, $\diml(\mu) \le s_0(\alpha)$. This observation, together with Theorem~\ref{thm:domibirkhoff}, implies
\begin{align*}
  s_n(\alpha) &= \sup\{\diml(\mu):\mu\in\MM_{\sigma^n}((\Sigma^{\DD}_n)^\N)\text{ and }\int_\Sigma\Phi \dd\mu=\alpha\} \\
  &\leq \sup\{\diml(\mu):\mu\in\MM_\sigma(\Sigma)\text{ and }\int_\Sigma\Phi \dd\mu=\alpha\} \\
  &\leq s_0(\alpha).
\end{align*}
By letting $n \to \infty$, we see that
\begin{equation*}
  s_0(\alpha) = \sup\{\diml(\mu):\mu\in\MM_\sigma(\Sigma)\text{ and }\int_\Sigma\Phi \dd\mu=\alpha\}.
\end{equation*}
The fact that the spectrum is continuous follows from Proposition \ref{cont}.
\end{proof}

\section{Lyapunov exponents}\label{sec:lyap}

Let us first study the Lyapunov spectrum for dominated systems.

\begin{proposition}\label{cor:lyapdomin}
Let $(A_1+v_1,\ldots,A_N+v_N)$ be an affine IFS on $\R^2$ satisfying the SOSC. If $(A_1,\ldots,A_N) \in GL_2(\R)^N$ is strongly irreducible and dominated such that the generated subgroup of the normalized matrices is non-compact, then
\begin{align*}
  \dimh(\pi E_\chi(\alpha)) &= \sup\{\diml(\mu) :\mu\in\MM_\sigma(\Sigma)\text{ and }\chi(\mu)=\alpha\} \\
  &= \min\left\{\frac{h_\mathrm{top}(E_\chi(\alpha))}{\alpha_1},1+\frac{h_\mathrm{top}(E_\chi(\alpha))-\alpha_1}{\alpha_2}\right\}
\end{align*}
for all $\alpha = (\alpha_1,\alpha_2) \in \PP(\chi)^o$.
\end{proposition}

\begin{proof}
Let $\Psi \colon \Sigma \to \R^2$ be as in \eqref{eq:Psi-def}. By Proposition~\ref{thm:BGprop}, $E_\chi(\alpha)=E_\Psi(\alpha)$ for all $\alpha=(\alpha_1,\alpha_2)\in\R^2$. Thus, the statement follows from Theorem~\ref{thm:domibirkhoff} and the following calculation.
Write $\Sigma_n(\alpha,\varepsilon)=\{\iii|_n\in\Sigma_n:|\alpha-\tfrac{1}{n}S_n\Psi(\iii)|<\varepsilon\}$. Observe that
\begin{align*}
  P(\langle s',\Psi\rangle+\langle q,\Psi-\alpha\rangle) &= \lim_{n\to\infty}\tfrac{1}{n}\log\sum_{\iii\in\Sigma_n}\exp(\langle q+s',S_n\Psi(\iii)\rangle-n\langle q,\alpha\rangle)\\
   &\geq \liminf_{n\to\infty}\tfrac{1}{n}\log\sum_{\iii\in\Sigma_n(\alpha,\varepsilon)}\exp(\langle q+s',S_n\Psi(\iii)\rangle-n\langle q,\alpha\rangle)\\
   &\geq -\langle s',\alpha\rangle-(|q|+|s'|)\varepsilon+\liminf_{n\to\infty}\tfrac{1}{n}\log\#\Sigma_n(\alpha,\varepsilon),
\end{align*}
for all $s',q \in \R^2$.
Recalling the definition of the topological entropy, as $\varepsilon>0$ is arbitrary, we get that
$$
\inf_{q\in\R^2}P(\log\varphi^s+\langle q,\Psi-\alpha\rangle)\geq
\begin{cases}
  h_\mathrm{top}(E_\chi(\alpha))-s\alpha_1, & \text{ if }0 \leq s < 1, \\
  h_\mathrm{top}(E_\chi(\alpha))-\alpha_1-(s-1)\alpha_2, & \text{ if }1 \leq s < 2, \\
  h_\mathrm{top}(E_\chi(\alpha))-(\alpha_1+\alpha_2)s/2, & \text{ if }2 \leq s < \infty,
\end{cases}
$$
by choosing $s'$ to be $(s,0)$, $(1,s-1)$, and $(s/2,s/2)$, respectively. Thus,
\begin{align*}
\min\biggl\{\frac{h_\mathrm{top}(E_\chi(\alpha))}{\alpha_1},&1+\frac{h_\mathrm{top}(E_\chi(\alpha))-\alpha_1}{\alpha_2}\biggr\}\\
&\leq \sup\{s \ge 0 : \inf_{q \in \R^2} P(\log\fii^s + \langle q,\Psi-\alpha \rangle) \ge 0\}.
\end{align*}
Since, for every $\mu\in\MM_\sigma(\Sigma)$ with $\chi(\mu)=\alpha$, we clearly have
$$
\diml(\mu)\leq\min\biggl\{\frac{h_\mathrm{top}(E_\chi(\alpha))}{\alpha_1},1+\frac{h_\mathrm{top}(E_\chi(\alpha))-\alpha_1}{\alpha_2}\biggr\},
$$
which completes the proof.
\end{proof}

Now we turn to general systems. Recall that
$$
\psi^q(\iii)=\|A_\iii\|^{q_1}\|A_\iii^{-1}\|^{-q_2}
$$
for all $q = (q_1,q_2) \in \R^2$, so $\psi^{s'(s)}=\fii^s$, where $s'$ is defined in \eqref{eq:s'}.

\begin{proposition}\label{thm:lyapupper}
  Let $(A_1+v_1,\ldots,A_N+v_N)$ be an affine IFS on $\R^2$. If $(A_1,\ldots,A_N) \in GL_2(\R)^N$ is irreducible, then
  $$
  \dimh(\pi E_{\chi}(\alpha))\leq \sup\{ s \ge 0 : \inf_{q \in \R^2}\{P(\log\psi^{s'(s)-q})-\langle q,\alpha \rangle\}\geq 0 \}
  $$
  for all $\alpha\in\PP(\chi)^o$, where $s'\colon\R_+ \to \R^2$ is defined in \eqref{eq:s'}.
\end{proposition}

\begin{proof}
The proof is almost identical to the proof of Proposition \ref{thm:birkhoffupper}.
Observe that
$$
E_{\chi}(\alpha)\subset\bigcap_{r=1}^\infty\bigcup_{n=1}^\infty\bigcap_{m=n}^\infty\bigcup_{\iii\in D_{m,r}}[\iii],
$$
where
$$
D_{m,r}=\{\iii\in\Sigma_m:|-\tfrac{1}{m}\log\psi^{(1,1)}(\jjj)-\alpha|<\tfrac{1}{r}\}.
$$
Note that $\psi^{(1,1)}(\jjj)$ is constant on $m$-th level cylinders.
Therefore, for every $\iii\in D_{m,r}$, we have
\begin{equation*}
-\frac{m|q|}{r}\leq\langle q,-\log\psi^{(1,1)}(\jjj)-m\alpha\rangle.
\end{equation*}
Let $s_0(\alpha)=\sup\{ s \ge 0 :  \inf_{q \in \R^2}\{P(\log\psi^{s'(s)-q})-\langle q,\alpha \rangle\} \geq 0 \}$ and choose $s > s_0(\alpha)$. Thus, there exists $q=q(\alpha,s)$ such that $P(\log\psi^{s'(s)-q})<\left\langle q,\alpha \right\rangle$. Let $\varepsilon>0$ be so small that there is $\gamma>0$ such that
$$
  \sum_{\iii\in \Sigma_n}\psi^{s'(s)-q}(\iii)<e^{n(\langle q,\alpha \rangle-\varepsilon)}
$$
for all $n\geq-\log\gamma$. We can find $0<\lambda<1$ such that $\norm{A_{\mathbf{i}}}\leq\lambda^{|\mathbf{i}|}$ for all $\mathbf{i}\in\Sigma^{*}$ and thus
\[
\psi^{s'(s+c)}(\iii)\leq \psi^{s'(s)}(\iii) e^{c|\iii|\log\lambda}
\]
for all $c \geq 0$, and hence
\begin{align*}
\mathcal{H}^{s-|q|/r\log\lambda}_\delta(\pi E_\chi(\alpha))&\leq\sum_{m=\lceil-\log\delta\rceil}^{\infty}\sum_{\iii\in D_{m,r}} \varphi^{s-|q|/r\log\lambda}(\iii) \\
&\leq\sum_{m=\lceil-\log\delta\rceil}^{\infty}\sum_{\iii\in D_{m,r}} \varphi^s(\iii)e^{-m|q|/r}\\
&\leq\sum_{m=\lceil-\log\delta\rceil}^{\infty}\sum_{\iii\in D_{m,r}} \varphi^s(\iii)e^{\langle q,-\log\psi^{(1,1)}(\jjj)-m\alpha\rangle}\\
&=\sum_{m=\lceil-\log\delta\rceil}^{\infty}e^{-m\langle q,\alpha\rangle}\sum_{\iii\in D_{m,r}} \psi^{s'(s)-q}(\iii)\\
&\leq\sum_{m=\lceil-\log\delta\rceil}^{\infty}e^{-m\varepsilon} \to 0
\end{align*}
as $\delta\to0$ for all $r\geq1$. Thus, $\dimh(\pi E_\chi(\alpha))\leq s-c|q|/r$. Since $r\geq1$ and $s>s_0(\alpha)$ were arbitrary, we get $\dimh(\pi E_\chi(\alpha))\leq s_0(\alpha)$.
\end{proof}

Let $\Sigma_n^\DD$ be as in \eqref{eq:dominsub} and let $\Phi_n\colon(\Sigma_n^\DD)^\N\to\R$ be a modified Lyapunov potential defined by
$$
\Psi_n(\iii_1\iii_2\cdots)=(-\log\|A_{\iii_1}|V_n(\sigma^n\iii)\|,-\log|\det(A_{\iii_1})|+\log\|A_{\iii_1}|V_n(\sigma^n\iii)\|),
$$
where $V_n\colon(\Sigma_n^\DD)^\N\to\RP$ is the subspace defined similarly as in \eqref{eq:oseledetsspace}.

\begin{lemma}\label{lem:pressapproxlyap}
We have
$$
  \lim_{n\to\infty}\tfrac{1}{n}P_{\DD,n}(\langle q-s'(s),\Psi_n\rangle) = P(\log\psi^{s'(s)-q})
$$
uniformly for all $(q,s)$ on any compact subset of $\R^2\times[0,\infty)$, where $s'\colon\R_+ \to \R^2$ is defined in \eqref{eq:s'}.
\end{lemma}

\begin{proof}
Since
\begin{align*}
P(\log\psi^{s'(s)-q})&=\lim_{m\to\infty}\tfrac{1}{m}\log\sum_{\iii\in\Sigma_m}\psi^{s'(s)-q}(\iii)\\
&=\lim_{k\to\infty}\tfrac{1}{nk}\log\sum_{\iii_1,\ldots,\iii_k\in\Sigma_{n}}\psi^{s'(s)-q}(\iii_1\cdots\iii_k)\\
&\geq\lim_{k\to\infty}\tfrac{1}{nk}\log\sum_{\iii_1,\ldots,\iii_k\in\Sigma_{n}^\DD}\psi^{s'(s)-q}(\iii_1\cdots\iii_k)\\
&=\lim_{k\to\infty}\tfrac{1}{nk}\log\sum_{\iii_1,\ldots,\iii_k\in\Sigma_{n}^\DD}\exp\biggl(\sup_{\jjj\in[\iii_1\cdots\iii_k]}\langle S_k^{(n)}\Psi_n(\jjj),q-s'(s)\rangle\biggr),
\end{align*}
we see that $P(\log\psi^{s'(s)-q})\geq\lim_{n\to\infty}\tfrac{1}{n}P_{\DD,n}(\langle q-s'(s),\Psi_n\rangle)$. On the other hand, by using the uniform domination of $(A_{\iii})_{\iii\in\Sigma_n^\DD}$ for every $n\in\N$, we have
\begin{align*}
\tfrac1n &P_{\DD,n}(\langle q-s'(s),\Psi_n\rangle)\\
&=\lim_{k\to\infty}\tfrac{1}{nk}\log\sum_{\iii_1,\ldots,\iii_k\in\Sigma_{n}^\DD}\exp\biggl(\sup_{\jjj\in[\iii_1\cdots\iii_k]}\langle S_k^{(n)}\Psi_n(\jjj),q-s'(s)\rangle\biggr)\\
&\geq\lim_{k\to\infty}\tfrac{1}{nk}\log\sum_{\iii_1,\ldots,\iii_k\in\Sigma_{n}^\DD}\exp\biggl(\sum_{\ell=1}^{k}\langle q-s'(s),-\log\psi^{(1,1)}(\iii_\ell)\rangle+C k\biggr)\\
&=\tfrac{1}{n}\log\sum_{\iii\in\Sigma_{n}^\DD}\exp\biggl(\langle q-s'(s),-\log\psi^{(1,1)}(\iii)\rangle\biggr)+\frac{C}{n}\\
&\geq\tfrac{1}{n}\log\sum_{\iii\in\Sigma_{n-2K}}\psi^{s'(s)-q}(\iii)+\frac{C'(K)}{n}.
\end{align*}
The statement follows by taking $n\to\infty$.
\end{proof}

We are now able to show the following result which is a stepping stone to Theorem \ref{thm:mainlyapunovint}.

\begin{theorem}\label{steppingstone}
Let $(A_1+v_1,\ldots,A_N+v_N)$ be an affine IFS on $\R^2$ satisfying the SOSC. If $(A_1,\ldots,A_N) \in GL_2(\R)^N$ is strongly irreducible such that the generated subgroup of the normalized matrices is non-compact, then
\begin{align*}
  \dimh(\pi E_\chi(\alpha)) &= \sup\{\diml(\mu) :\mu\in\MM_\sigma(\Sigma)\text{ and }\chi(\mu)=\alpha\} \\
  &= \sup\{\diml(\mu) : \mu\in\EE_\sigma(\Sigma)\text{ and }\chi(\mu)=\alpha\} \\
  &= \sup\{s\geq 0:\inf_{q\in\R^2}\{P(\log\psi^{s'(s)-q})-\langle q,\alpha\rangle\}\geq 0\} \\
  &=\min\biggl\{\frac{h_\mathrm{top}(E_\chi(\alpha))}{\alpha_1},1+\frac{h_\mathrm{top}(E_\chi(\alpha))-\alpha_1}{\alpha_2}\biggr\}
\end{align*}
for all $\alpha=(\alpha_1,\alpha_2)\in\PP(\chi)^o \subset \R^2$.
\end{theorem}

\begin{proof}
Write $\Sigma_n(\alpha,\varepsilon)=\{\iii\in\Sigma_n:|-\tfrac{1}{n}\log\psi^{(1,1)}(\iii)-\alpha|<\varepsilon\}$. Observe that
\begin{align*}
  P(\log\psi^{s'-q}) &\geq \liminf_{n\to\infty}\tfrac{1}{n}\log\sum_{\iii\in\Sigma_n(\alpha,\varepsilon)}\psi^{s'-q}(\iii)\\
   &\geq -\langle s',\alpha\rangle-C(|s'|+|q|)\varepsilon+\liminf_{n\to\infty}\tfrac{1}{n}\log\#\Sigma_n(\alpha,\varepsilon)
\end{align*}
for all $s',q \in \R^2$.
Recalling the definition of the topological entropy, as $\varepsilon>0$ is arbitrary, we get that
$$
\inf_{q\in\R^2}P(\log\psi^{s'(s)-q})\geq
\begin{cases}
h_\mathrm{top}(E_\chi(\alpha))-s\alpha_1, & \text{ if }0 \leq s < 1, \\
h_\mathrm{top}(E_\chi(\alpha))-\alpha_1-(s-1)\alpha_2, & \text{ if }1 \leq s < 2, \\
h_\mathrm{top}(E_\chi(\alpha))-(\alpha_1+\alpha_2)s/2, & \text{ if }2 \leq s < \infty.
\end{cases}
$$
Thus,
\begin{align*}
\min\biggl\{\frac{h_\mathrm{top}(E_\chi(\alpha))}{\alpha_1},&1+\frac{h_\mathrm{top}(E_\chi(\alpha))-\alpha_1}{\alpha_2}\biggr\} \\
&\leq \sup\{s\geq 0:\inf_{q\in\R^2}\{P(\log\psi^{s'(s)-q})-\langle q,\alpha\rangle\}\geq 0\}.
\end{align*}
We clearly have $E_\chi(\alpha)\supset E_\chi^{\DD,n}(\alpha)$ and $h_\mathrm{top}(E_\chi(\alpha),\sigma)\geq \tfrac{1}{n}h_\mathrm{top}(E_\chi^{\DD,n}(\alpha),\sigma^n)$, where
$$
E_\chi^{\DD,n}(\alpha)=\{\iii\in(\Sigma_n^\DD)^\N:\lim_{k\to\infty}\tfrac{1}{n}S_k^{(n)}\Phi_n(\iii)=\alpha\}.
$$
Observe that, by Proposition~\ref{cor:lyapdomin} and Theorem~\ref{thm:domibirkhoff},
\begin{align*}
 \dimh(\pi E_\chi^{\DD,n}(\alpha))&=\sup\{\diml(\mu'):\mu'\in\MM_{\sigma^n}((\Sigma_n^\DD)^\N)\text{ and }\chi(\mu',\sigma^n)=n\alpha\}\\
 &=\sup\{ s \ge 0 : \inf_{q \in \R^2}(P_{\DD,n}(\langle q-s'(s),\Psi_n\rangle)-\langle q,n\alpha \rangle)\geq 0 \} \\
 &=\min\biggl\{\frac{h_\mathrm{top}(E_\chi^{\DD,n}(\alpha),\sigma^n)}{n\alpha_1},1+\frac{h_\mathrm{top}(E_\chi^{\DD,n}(\alpha),\sigma^n)-n\alpha_1}{n\alpha_2}\biggr\},
\end{align*}
and, by Lemma~\ref{lem:pressapproxlyap},
\begin{align*}
\sup\{ s \ge 0 : \inf_{q \in \R^2}&(P_{\DD,n}(\langle q-s'(s),\Psi_n\rangle)-\langle q,n\alpha \rangle)\geq 0 \} \\
&\to \sup\{s\geq 0:\inf_{q\in\R^2}\{P(\log\psi^{s'(s)-q})-\langle q,\alpha\rangle\}\geq 0\}
\end{align*}
as $n\to\infty$. Finally, we note that for any $\mu'\in\MM_{\sigma^n}((\Sigma_n^\DD)^\N)$ there exists $\mu\in\MM_{\sigma}(\Sigma)$ such that $\mu=\frac{1}{n}\sum_{k=0}^{n-1}\mu' \circ \sigma^{-k}$, $n\chi(\mu,\sigma)=\chi(\mu',\sigma^n)$, and $\diml(\mu)=\diml(\mu')$. This and the variational principle argument used in the proof of Theorem \ref{thm:mainbirkhoffint} complete the proof.
\end{proof}

\section{Boundaries}\label{sec:lyapdiag}

To complete the proof of Theorem \ref{thm:mainlyapunovint}, we need to show that $\PP(\chi)$ is closed and convex, and the continuity of the spectrum.

\begin{proposition} \label{prop:wmeasbirk}
Let $\alpha_k\to \alpha$ be a converging sequence of points in $\PP(\Phi)$ and, for each $k$, let $\nu_k$ be an ergodic measure so that $\nu_k(E_\Phi(\alpha_k))=1$. Then $h_{\rm top}(E_\Phi(\alpha)) \geq \limsup_{k\to\infty} h(\nu_k)$.
\end{proposition}

\begin{proof}
Consider a fast increasing sequence of integers $(m_k)$. We define a measure $\nu$ by setting on level $m_\ell$ cylinders $\nu([\iii_1\cdots\iii_\ell]) = \nu_1([\iii_1])\cdots\nu_\ell([\iii_\ell])$ for all $\iii_1 \in \Sigma_{m_1}$ and $\iii_k \in \Sigma_{m_k-m_{k-1}}$ for all $k \in \{2,\ldots,\ell\}$. Write $s=\limsup_{k\to\infty} h(\nu_k)$. We claim that if $(m_k)$ grows quickly enough, then for all $t<s$ and for some $C>0$,
$$
\nu(\{\iii\in E_{\Phi}(\alpha):\nu([\iii|_n])\leq C e^{-nt}\})>0.
$$
From this it follows, by Takens and Verbitskiy \cite[Theorem 3.6]{TV}, that $h_{\text{top}}(E_{\Phi}(\alpha))\geq s$.
The proof of the claim is virtually identical as (in fact, simpler than) the proof of \cite[Proposition 9]{GelRam:09}, but as there exist formal differences (instead of ergodic measures, the proof there was given for Gibbs measures) we will sketch the proof.

The main ingredient is the following statement: Let $\mu$ be an ergodic measure and $\Phi=(\Phi_1,\ldots,\Phi_M) \colon \Sigma \to \R^M$ be a continuous potential. Then for every $\varepsilon>0$ there exists $L>0$ such that for any $n>0$ the union of cylinders $[\iii]$, where each $\iii\in\Sigma_n$ satisfies
\begin{equation} \label{eqn:wmeas1}
  \ell\biggl(-\varepsilon + \int_\Sigma \Phi_k \dd\mu\biggr) - L \leq S_\ell \Phi_k(\jjj) \leq \ell\biggl(\varepsilon + \int_\Sigma \Phi_k \dd\mu\biggr) + L
\end{equation}
for all $\jjj\in [\iii]$, $\ell \in \{1,\ldots,n\}$, and $k\in\{1,\ldots,M\}$, and
\begin{equation} \label{eqn:wmeas2}
  L^{-1} e^{-\ell(h(\mu)+\varepsilon)} \leq \mu([\iii]) \leq Le^{-\ell(h(\mu)-\varepsilon)}
\end{equation}
for all $\ell \in \{1,\ldots,n\}$, has $\mu$-measure at least $1-\eps$. This follows from Birkhoff and Shannon-McMillan-Breiman Theorems together with Egorov Theorem.

Let $(\varepsilon_k)_{k \in \N}$ be a sequence such that $\varepsilon_k\downarrow 0$ as $k\to\infty$ and $\prod_{i=1}^{\infty} (1-\varepsilon_i)>0$. We apply the above statement to each $\nu_k$ with the corresponding $\varepsilon_k$. Then, for every $\ell \in \{m_k+1,\ldots, m_{k+1}\}$, we have
\begin{align*}
  S_\ell \Phi(\jjj) &= m_k \int_\Sigma \Phi \dd\nu_k + (\ell - m_k) \int_\Sigma \Phi \dd\nu_{k+1} \\
  &\qquad+ O(m_{k-1}, L(\nu_{k+1},\varepsilon_{k+1}), \varepsilon_k (m_k-m_{k-1}), \varepsilon_{k+1} (\ell - m_k))
\end{align*}
for all $\jjj \in [\iii]$ and
\begin{align*}
  -\log \nu([\iii]) &= m_k h(\nu_k) + (\ell-m_k) h(\nu_{k+1}) \\
  &\qquad+ O(m_{k-1}, L(\nu_{k+1},\varepsilon_{k+1}), \varepsilon_k (m_k-m_{k-1}), \varepsilon_{k+1} (\ell - m_k))
\end{align*}
for cylinders $[\iii]$ containing in a set of $\nu$-measure at least $\prod_{i=1}^{k+1} (1-\varepsilon_i)$. Thus, for $(m_k)$ growing sufficiently fast, we get the claim (using \cite[Theorem 3.6]{TV} for the entropy part of the claim).
\end{proof}

A similar statement holds for Lyapunov exponents.

\begin{proposition} \label{prop:wmeaslyap}
Let $\alpha_k\to \alpha$ be a converging sequence of points in $\PP(\chi)$ and, for each $k$, let $\mu_k$ be an ergodic measure so that $\mu_k(E_\chi(\alpha_k))=1$. Then $h_{\rm top}(E_\chi(\alpha)) \geq \limsup_{k\to\infty} h(\mu_k)$.
\end{proposition}

The main obstacle in repeating the proof of the previous proposition is that the singular value is not multiplicative. We can, however, use Lemma \ref{lem:consdomi} to transfer the argument to a dominated cocycle setting, where the singular values are almost multiplicative and the same argument as in Proposition \ref{prop:wmeasbirk} will work.

To prove Proposition \ref{prop:wmeaslyap}, we begin with an approximation argument.

\begin{lemma} \label{lem:skeleton}
%Let $(A_1,\ldots,A_N) \in GL_2(\R)^N$ be strongly irreducible such that the generated subgroup of the normalized matrices is non-compact.
If $\mu$ is an ergodic measure supported on $E_\chi(\alpha)$, then for every $\eps>0$ there exist arbitrarily large $n\in\N$ such that the set
\begin{align*}
  \Omega_n(\alpha,\varepsilon)=\{\iii|_n\in\Sigma_n : \;&|-\tfrac1n \log \norm{A_{\iii}}-\alpha_1|<\varepsilon, \\ &|-\tfrac1n \log \norm{A_{\iii}^{-1}}^{-1}-\alpha_2|<\varepsilon, \text{ and }A_{\iii}\CC\subset \BB^o\},
\end{align*}
where $\BB$ and $\CC$ are multicones defined in Lemma \ref{lem:consdomi}, has at least $e^{n(h(\mu)-\eps)}$ elements.
\end{lemma}

\begin{proof}
Consider first the set $\Sigma_{n-2K}(\alpha,\varepsilon/2)$, where $K$ is as in Lemma \ref{lem:consdomi} and $\Sigma_n(\alpha,\varepsilon)=\{\iii\in\Sigma:|\alpha-\tfrac{1}{n}S_n\Psi(\iii)|<\varepsilon\}$, and $\Psi \colon \Sigma \to \R^2$ is as in \eqref{eq:Psi-def}. For every $\varepsilon>0$, there exist arbitrarily large $n\in\N$ such that this set has at least $e^{(n-2K)(h(\mu)-\varepsilon)}$ elements. Indeed, if this statement was not true, then for some $N \in \N$ we could cover all the points $\iii \in \Sigma$ with $(\chi_1(\iii),\chi_2(\iii)) = \alpha$ (that is, $\mu$-almost every point) with a collection of cylinders containing $e^{n(h(\mu)-\varepsilon)}$ cylinders of level $n$ for all $n>N$. This would imply that $h(\mu)\leq h(\mu)-\varepsilon$.

Now, Lemma \ref{lem:consdomi} lets us find for every word $\iii\in\Sigma_{n-2K}(\alpha,\varepsilon/2)$ a prefix $\jjj_1$ and a suffix $\jjj_2$ such that $A_{\jjj_1\iii\jjj_2}\CC\subset\BB^o$. At the same time,
\[
\norm{A_{\iii}} H^{-2K} \leq \norm{A_{\jjj_1\iii\jjj_2}} \leq \norm{A_{\iii}} H^{2K}
\]
and similarly for $\norm{A_{\jjj_1\iii\jjj_2}^{-1}}^{-1}$, where $H=\max_i \max\{\norm{A_i}, \norm{A_i^{-1}}^{-1}\}$. Hence, for $n$ large enough, if $\iii\in\Sigma_{n-2K}(\alpha,\varepsilon/2)$, then $\jjj_1\iii\jjj_2\in\Omega_n(\alpha,\varepsilon)$.
\end{proof}

\begin{corollary} \label{cor:skel}
Let $\mu$ be an ergodic measure supported on $E_\chi(\alpha)$. For $\eps>0$, let $n \in \N$ be such that $\Omega_n(\alpha,\eps)$ satisfies the claim in Lemma \ref{lem:skeleton}. Let $\nu$ be an $n$-step Bernoulli measure generated by the words in $\Omega_n(\alpha,\epsilon)$ with equally distributed probability and let $Z$ be as in Lemma \ref{lem:dominbound}. Then $h(\nu) \geq h(\mu)-\eps$ and
\begin{align*}
  \alpha_1 - \eps - Zn^{-1} &\leq -\tfrac 1\ell \log \norm{A_{\iii|_\ell}} \leq \alpha_1 + \eps + Zn^{-1}, \\
  \alpha_2 - \eps - Zn^{-1} &\leq -\tfrac 1\ell \log \norm{A_{\iii|_\ell}^{-1}}^{-1} \leq \alpha_2 + \eps + Zn^{-1},
\end{align*}
for $\nu$-almost all $\iii \in \Sigma$ and for all $\ell =kn$.
\end{corollary}

\begin{proof}
The bounds on the singular values follow from Lemma \ref{lem:dominbound}, the entropy estimation follows from the bound on the size of $\Omega_n(\alpha,\eps)$.
\end{proof}

\begin{proof}[Proof of Proposition \ref{prop:wmeaslyap}]
We can now replace the measures $\mu_k$ by $n_k$-step Bernoulli measures $\nu_k$ given by Corollary \ref{cor:skel}, taking care that $\eps_k\downarrow 0$ and $n_k\to\infty$. We repeat the proof of Proposition \ref{prop:wmeasbirk}. The inequality \eqref{eqn:wmeas2} can be proved as before. The main problem is to obtain \eqref{eqn:wmeas1} -- we cannot apply the Birkhoff Theorem anymore since the singular values in a matrix cocycle are not multiplicative. For each measure $\nu_k$, Corollary \ref{cor:skel} gives us
\begin{align*}
  \log \norm{A_{\iii|_\ell}} &= \ell \chi_1(\mu_k) + O(\ell \epsilon_k, \ell Z/n_k, n_k \log H), \\
  \log \norm{A_{\iii|_\ell}^{-1}}^{-1} &= \ell \chi_2(\mu_k) + O(\ell \epsilon_k, \ell Z/n_k, n_k \log H),
\end{align*}
for $\nu_k$-almost all $\iii \in \Sigma$ and for all $\ell\in\N$. Thus we obtain \eqref{eqn:wmeas1} if $\varepsilon_k$ in the proof of Proposition \ref{prop:wmeasbirk} is replaced by $\eps_k + Z/n_k$. Note that we can choose $\eps_k$ and $n_k$ in the beginning so that $\eps_k + Z/n_k$ is as small as we wish, so this will not cause any problems.

The rest of the proof is virtually unchanged. We construct the measure $\nu$, prove that a positive part of this measure lives on $E_\chi(\alpha)$ (by using Lemma \ref{lem:dominbound} again), and show that it has entropy at least $\limsup_{k \to \infty} h(\nu_k)=\limsup_{k \to \infty} h(\mu_k)$.
\end{proof}

We can now apply Proposition \ref{prop:wmeaslyap} for the Hausdorff dimension.

\begin{proposition} \label{prop:w-dim}
Let $(\alpha_1^k,\alpha_2^k) = \alpha^k \to \alpha=(\alpha_1,\alpha_2)$ be a converging sequence of points in $\PP(\chi)$ such that $\alpha_1=\alpha_2$. Then
$$
\dimh(\pi E_\chi(\alpha)) \geq \limsup_{k\to\infty} \dimh(\pi E_\chi(\alpha^k))= \frac 1 {\alpha_2} \limsup_{k\to\infty} h_{\rm top}(E_\chi(\alpha^k)).
$$
\end{proposition}

\begin{proof}
Note first that, by Theorem \ref{thm:mainlyapunovint}, we can pick a sequence of ergodic measures $\mu_k$ supported on $E_\chi(\alpha_k)$ and $\eps_k \downarrow 0$ such that
\[
\diml(\mu_k) \geq \dim_H(\pi E_\chi(\alpha_k)) - \eps_k.
\]
By the definition of the Lyapunov dimension, we have
\[
h(\mu_k) \geq \alpha^k_2 \diml(\mu_k).
\]
Hence, by Proposition \ref{prop:wmeaslyap} we have a measure $\mu$ supported on $\pi E_\chi(\alpha)$ satisfying
\[
h(\mu) \geq \limsup_{k\to\infty} \alpha^k_2 \dimh(\pi E_\chi(\alpha^k)) = \alpha_2 \limsup_{k\to\infty} \dimh(\pi E_\chi(\alpha^k)).
\]
By the Bowen's definition of entropy, as $\alpha_1=\alpha_2$, we have $\dimh(\mu) = h(\mu)/\alpha_2$. This is what we wanted.
\end{proof}

The following proposition finishes the proof of Theorem \ref{thm:mainlyapunovint}. We note that the concavity of a function defined on a convex set implies the continuity of the function in the interior and on the flat portions of the boundary, and that the continuity of the entropy under the change of the Lyapunov exponent implies the continuity of the Lyapunov dimension.

\begin{proposition}\label{propconcave}
The set $\PP(\chi)$ is compact and convex, and the function $\alpha\mapsto h_{\rm top}(E_\chi(\alpha))$ is concave on
$$
\PP(\chi)^o\cup(\PP(\chi)\cap\{(\alpha_1,\alpha_2)\in\R^2:\alpha_1=\alpha_2\}.
$$
\end{proposition}

\begin{proof}
The compactness of $\PP(\chi)$ follows from Proposition \ref{prop:wmeaslyap}, the other properties will be proven together.
Let $\alpha, \beta \in \PP(\chi)$ and choose $\ell \in (0,1)$. We have to show that
\[
h_{\rm top}\left(E_\chi(\ell \alpha + (1-\ell)\beta)\right) \geq \ell h_{\rm top}(E_\chi(\alpha)) + (1-\ell) h_{\rm top}(E_\chi(\beta)).
\]
For a fixed large $r$ we define
\[
D_{m,r, \alpha}=\{\iii\in\Sigma_m:|-\tfrac{1}{m}\log\psi^{(1,1)}(\iii)-\alpha|<\tfrac{1}{r}\}
\]
and similarly $D_{m,r,\beta}$. By the definition of the topological entropy, there exist two sequences $(m_i)_{i \in \N}$ and $(n_j)_{j \in \N}$ of positive integers such that
\begin{align*}
  |D_{m_i,r,\alpha}| &> e^{m_i(h_{\rm top}(E_\chi(\alpha))-1/r)}, \\
  |D_{n_j,r,\beta}| &> e^{n_j(h_{\rm top}(E_\chi(\beta))-1/r)},
\end{align*}
for all $i,j \in \N$.

Fix $i,j \in \N$. By Lemma \ref{lem:consdomi}, there exist multicones $\BB$ and $\CC$ such that $\BB\subset \CC^o$ and for every word $\iii\in D_{m_i,r, \alpha}$ we can add a prefix $\jjj_1$ and a suffix $\jjj_2$ of fixed length $K$ such that the resulting word $\jjj=\jjj_1 \iii \jjj_2$ satisfies $A_\jjj\CC \subset \BB^o$. We denote this set of words by $D'_{m_i, r, \alpha}$. The same statement, with the same choice of $\BB$, $\CC$, and $K$, holds for $D_{n_j,r,\beta}$. Note that, as the prefixes are of fixed length, different $\iii$'s produce different $\jjj$'s.
For $\jjj\in D'_{m_i,r, \alpha}$, we have
\[
\biggl|-\frac{1}{m_i+2K}\log\psi^{(1,1)}(\jjj)-\alpha\biggr|<\frac{1}{r} + O\biggl(\frac {2K} {m_i}\biggr) < \frac 2r
\]
for $i$ large enough. A similar estimate holds for $\jjj \in D'_{n_j,r,\beta}$.

Consider now a Sturmian sequence $\iii=\iii(i,j,\ell)$ with average $n_j \ell /(n_j \ell + m_i (1-\ell))$. Let $Q$ be the set of all possible sequences in $\Sigma$ obtained from $\iii$ by replacing all the $1$'s by some elements of $D'_{m_i, r, \alpha}$ and all the $0$'s by some elements of $D'_{n_j, r, \beta}$. Let $\mu$ be the measure on $Q$ obtained by assuming that in this construction all the words are used with the same probability. It follows that
\begin{align*}
h_{\rm top}(Q) &\geq h(\mu) \\
 &\geq \biggl(1-\frac K {\min(m_i, n_j)}\biggr) \biggl(\ell h_{\rm top}(E_\chi(\alpha)) + (1-\ell) h_{\rm top}(E_\chi(\beta)) - \frac 1r\biggr) \\
 &\geq \ell h_{\rm top}(E_\chi(\alpha)) + (1-\ell) h_{\rm top}(E_\chi(\beta)) - \frac 2r
\end{align*}
for all $i,j \in \N$ large enough. By Lemma \ref{lem:dominbound}, we can also write
\[
|\chi(\mu) - \ell \alpha - (1-\ell) \beta| \leq \frac 2r + \frac Z {\max(m_i, n_j)} \leq \frac 3r
\]
for all $i,j \in \N$ large enough. This means that there exists $\gamma_r$ with $|\gamma_r - \ell \alpha - (1-\ell)\beta| < 3/r$ such that
\[
h_{\rm top}(E_\chi(\gamma_r)) \geq \ell h_{\rm top} E_\chi(\alpha) + (1-\ell) h_{\rm top} E_\chi(\beta) - \frac 2r.
\]
We now pass with $r$ to infinity, repeating this procedure. By Proposition~\ref{prop:wmeaslyap}, we obtain a sequence $\gamma_r \to \ell \alpha + (1-\ell)\beta$ such that
\begin{align*}
  \ell h_{\rm top}(E_\chi(\alpha)) + (1-\ell) h_{\rm top}(E_\chi(\beta)) &\leq \limsup_{r\to\infty}h_{\rm top}(E_\chi(\gamma_r)) \\
  &\leq h_{\rm top}(E_\chi(\ell\alpha + (1-\ell) \beta)).
\end{align*}
This is what we wanted to show.
\end{proof}

\subsection*{Acknowledgements}

This work began while all the authors were participating the program ``Fractal Geometry and Dynamics'' at the Mittag-Leffler Institut in November and December, 2017. We thank the hospitality of the Mittag-Leffler Institut, as well as the Budapest University of Technology and the University of Bristol, where the research took place. B\'ar\'any acknowledges support from the grants OTKA K123782, NKFI PD123970, and the J\'anos Bolyai Research Scholarship of the Hungarian Academy of Sciences. K\"aenm\"aki was supported by the Finnish Center of Excellence in Analysis and Dynamics Research, the Finnish Academy of Science and Letters, and the V\"ais\"al\"a Foundation. Rams was supported by National Science Centre grant 2014/13/B/ST1/01033 (Poland). The authors also express their gratitude to the anonymous referee for his/her valuable comments and suggestions for future studies.

%\bibliographystyle{abbrv}
%\bibliography{Bibliography}

%%%%%%%%%%%%%%%%%%%%%%%%%%%%%%%%%%%%%%%%%%%%

\end{document}